\documentclass{amsart} 

\usepackage{amssymb, amsmath}
\usepackage{amsfonts}
\usepackage{blindtext}
\usepackage{color}
\usepackage{comment}
\usepackage{epsfig}
\usepackage{float}
\usepackage{graphicx}
\usepackage{array,makecell}
\usepackage{url}
\newtheorem{theorem}{Theorem}[section]
\newtheorem{lemma}[theorem]{Lemma}
\newtheorem{proposition}[theorem]{Proposition}
\newtheorem{corollary}[theorem]{Corollary}

\theoremstyle{definition}
\newtheorem{definition}[theorem]{Definition}

\theoremstyle{remark}
\newtheorem{remark}[theorem]{Remark}

\newcommand{\cB}{\mathcal B}
\newcommand{\cC}{\mathcal C}

\newcommand{\cW}{\mathcal W}

\newcommand{\R}{\mathbb R}

\newcommand{\dist}{\text{\rm dist}}

\def\be{\begin{equation}}
\def\ee{\end{equation}}

\def\R{\mathbb R}

\def\p{\partial}
\def\a{\alpha}

\def\z{\zeta}

\renewcommand\tilde{\widetilde}
\numberwithin{equation}{section} 
\numberwithin{figure}{section}
\vspace{1cm}

\def\O{\Omega}
\DeclareMathOperator{\diam}{{\rm diam}}

\DeclareMathOperator{\conv}{conv}

\newcommand{\gain}{\text{\rm gain}}
\newcommand{\loss}{\text{\rm loss}}

\makeatletter
\newcommand{\opnorm}{\@ifstar\@opnorms\@opnorm}
\newcommand{\@opnorms}[1]{%
  \left|\mkern-1.5mu\left|\mkern-1.5mu\left|
   #1
  \right|\mkern-1.5mu\right|\mkern-1.5mu\right|
}
\newcommand{\@opnorm}[2][]{%
  \mathopen{#1|\mkern-1.5mu#1|\mkern-1.5mu#1|}
  #2
  \mathclose{#1|\mkern-1.5mu#1|\mkern-1.5mu#1|}
}
\makeatother

\allowdisplaybreaks

\begin{document}
\bibliographystyle{siam}

\title[]
{On the existence and regularity of  weakly nonlinear stationary Boltzmann equations : a Fredholm alternative approach}

\author[I-K.~Chen, C.-H.~Hsia and D.~Kawagoe]{I-KUN CHEN, CHUN-HSIUNG HSIA AND DAISUKE KAWAGOE}

\date{\today}

\begin{abstract}
The celebrated Fredholm alternative theorem works for the setting of identity compact operators. This idea has been widely used to solve linear partial differential equations \cite{Evans}.  
In this article, we demonstrate a generalized  Fredholm theory in the setting of identity power compact operators, which was suggested in Cercignani and Palczewski \cite{CP} to solve the existence of the stationary Boltzmann equation in a slab domain. We carry out the detailed analysis based on this generalized Fredholm theory to prove the existence theory of the stationary Boltzmann equation in bounded three-dimensional convex domains. To prove that the integral form of the linearized Boltzmann equation satisfies the identity power compact setting requires the regularizing effect of the solution operators. Once the existence and regularity theories for the linear case are  established, with suitable bilinear estimates, the nonlinear existence theory is accomplished.

\end{abstract}

\maketitle

\tableofcontents

\section{Introduction} \label{sec:intro}
Let $F = F(x, \z)$ be the density distribution function of gas particles with position $x \in \Omega$ and velocity $\z \in \mathbb R^3$.
The stationary Boltzmann equation reads
\begin{equation} \label{BVP_non1}
\z \cdot \nabla_x F = Q(F, F) \,\, \mbox{ in } \Omega \times \mathbb{R}^{3}.
\end{equation}
The collision operator $Q$  is defined as
\begin{equation}
\label{collision1}
Q(F, G) := \int_{\R^3} \int_0^{2\pi} \int_0^{\frac{\pi}{2}} [F(\z') G(\z_*') - F(\z) G(\z_*)] B(|\z - \z_*|, \theta)\,d\theta d\phi d\z_*,
\end{equation}
where
\begin{align*}
\z' :=& \z + ((\z_* - \z) \cdot \omega) \omega, \quad \z_*' := \z_* - ((\z_* - \z) \cdot \omega) \omega,\\
\omega :=& \cos \theta \frac{\z_* - \z}{|\z_* - \z|} + (\sin \theta \cos \phi) e_2 + (\sin \theta \sin \phi) e_3.
\end{align*}
Here, $0 \leq \phi \leq 2\pi$, $0 \leq \theta \leq \pi/2$, and $e_2$ and $e_3$ are unit vectors such that the pair $\{ (\z_* - \z)/|\z_* - \z|, e_2, e_3 \}$ forms an orthonormal basis in $\R^3$. 

Various kinds of boundary conditions for the problem \eqref{BVP_non1} have been considered by many authors for different research interests.
For example, the diffuse reflection boundary condition can be implemented by
\begin{equation}
\label{diffuse}
F(x, \z) =\mu_{\theta} \int_{\Gamma^{+}} F(x, \z') (n(x) \cdot \z' ) d \z'   \text{ for }  (x, \z) \in \Gamma^{-},
\end{equation}
where
$ \mu_{\theta} = \frac{1}{2 \pi \theta^2(x)} e^{- \frac{|\z|^2}{2 \theta(x)}}$, $n(x)$ is the outward unit normal at $x \in \partial \Omega$  and 
\[
\Gamma^\pm := \{ (x, \z) \in \partial \Omega \times \mathbb{R}^{3} \mid \pm n(x) \cdot \z > 0 \}.
\]
In case where $\Omega$ is a slab domain, $L^1$ solutions of \eqref{BVP_non1}-\eqref{diffuse} are obtained in Arkeryd and Nouri \cite{Arkeryd-N-2000}. In case where $\Omega$ is a convex domain, Guiraud et al. proved the existence and uniqueness of the solution to the boundary value problem \eqref{BVP_non1}-\eqref{diffuse}. For the case where $\Omega$ is a general bounded connected subset of $ \mathbb R^d$ with $d =1,2,3$,  Esposito et al.\cite{GuoKim} considered  a penalization procedure and used the concept of stochastic cycles and Vidav's iterations (see, also \cite{Vidav}) to prove the existence and dynamical stability of the small-amplitude non-Maxwellian solution for the hard potential case. For the soft potential case, Duan et al. \cite{Duan2019} introduced a new mild formulation of solutions to the  problem \eqref{BVP_non1}-\eqref{diffuse} along the speeded backward bicharacteristic and carried out the penalization procedure as in \cite{GuoKim} to prove the existence and dynamical stability of  the small-amplitude non-Maxwellian solution. Along this line, in \cite{WuWang}, Wu and Wang establish existence for incoming boundary problem and extend the H\"{o}lder regularity result of \cite{RegularChen} to  soft cutoff  potential for full Boltzmann equation.   For the details of subtleties of the hard potential case and the soft potential case, see \cite{GuoKim}, \cite{Duan2019}, and \cite{WuWang} respectively.
 
The incoming boundary value problem of the stationary Boltzmann equation reads
\begin{equation} \label{BVP_non}
\begin{cases}
\z \cdot \nabla_x F = Q(F, F) &\mbox{ in } \Omega \times \mathbb{R}^{3},\\
F = F_0 &\mbox{ on } \Gamma^-.
\end{cases}
\end{equation}
In \cite{Arkeryd-N-2002}, Arkeryd and Nouri studied the $L^1$ solutions to the problem \eqref{BVP_non} for a general domain $\Omega$ with a small velocity cutoff. For the case where $\Omega$ is a slab domain, Pao \cite{Pao} and Cercignani and Palczewski \cite{CP} established the existence theory for \eqref{BVP_non} in a weighted $L^{\infty}$ space by the implicit function theorem. In particular,  Cercignani and Palczewski \cite{CP}  used  a generalized form of Fredholm alternative to carry out the existence theory for the linearized problem of \eqref{BVP_non}.  Their linear existence theory together with the bilinear estimates fulfill the hypotheses to apply the implicit function in the weigted $L^{\infty}$ space. Later on, van der Mee \cite{van der Mee} tried to extend the idea of Cercignani and Palczewski \cite{CP} to a general bounded spatial domain. But the result obtained in \cite{van der Mee} only works for one dimensional spatial domains, see the mathematical review on MathSciNet made by Nina B. Maslova.

The generalized form of the Fredholm alternative theorem that  Cercignani and Palczewski   used in \cite{CP} can be formulated as follows.

\begin{theorem} \label{Fredholm-powerCompact}
Let $T$ be a bounded linear operator on a Banach space $X$. Suppose $T$ is power compact. Let $\lambda$ be a non-zero complex number which is not an eigenvalue of $T$. Then the operator $T - \lambda$ has the bounded inverse $(T - \lambda)^{-1}$ on $X$.
\end{theorem}
\noindent For readers' convenience, we  provide a detailed proof of Theorem~\ref{Fredholm-powerCompact} in Section 4. 

In this article, we consider solutions to \eqref{BVP_non} near the global Maxwellian 
\[
M(\z) := \pi^{-\frac{3}{2}} e^{-|\z|^2}.
\]
Set $F = M + M^{\frac{1}{2}} f$ and $F_0 = M + M^{\frac{1}{2}} f_0$. The boundary value problem \eqref{BVP_non} is rewritten as:
\begin{equation} \label{BVP_non_red}
\begin{cases}
\z \cdot \nabla_x f = Lf + \Gamma(f, f) &\mbox{ in } \Omega \times \mathbb{R}^{3},\\
f = f_0 &\mbox{ on } \Gamma^-,
\end{cases}
\end{equation}
where
\begin{align*}
&Lh := M^{-\frac{1}{2}} \left( Q(M, M^{\frac{1}{2}} h) + Q(M^{\frac{1}{2}} h, M) \right),\\
&\Gamma(h_1, h_2) := M^{-\frac{1}{2}} Q(M^{\frac{1}{2}} h_1, M^{\frac{1}{2}} h_2).
\end{align*}
Throughout this article, for the collision operator $Q$ defined in \eqref{collision1}, we consider the following cross section $B$: 
\begin{equation} \label{assumption_B1}
B(|\z - \z_*|,\theta) = B_0 |\z - \z_*|^{\gamma}\sin{\theta}\cos{\theta}, \,\, \qquad  0 \leq \gamma \leq 1,
\end{equation}
for some $B_0 > 0$. The range of $\gamma$ corresponds to the hard sphere model, the hard potential, and the Maxwell potential. Under the setting \eqref{assumption_B1}, the linearized collision operator $L$ can be decomposed as
\[
Lh (x, \z) = - \nu(|\z|) h(x, \z) + Kh (x, \z), 
\]
where 
\begin{equation} \label{K}
Kh (x, \z):=\int_{\mathbb{R}^{3}}k(\z, \z_*) h(x, \z_*)\,d\z_*. 
\end{equation}
The properties  of the operators $L$ and $K$ shall be addressed in Section 2.

In this article, we extend the idea of Cercignani and Palczewski \cite{CP} to bounded three dimensional convex domains. We first build up the regularizing effects of the solution operators of the inhomogeneous linear problem and prove that the regularity assure power compactness that we need when applying the Fredholm theory. We then bootstrap  the regularity for the inhomogeneous linear problem to satisfactory smoothness, see Theorem \ref{lem:sol_dz_lin}. Finally, with the help of the bilinear estimate (Lemma~\ref{lem88}), we obtain differentiable solutions for the nonlinear problem, see Theorem~\ref{main theorem 2}.

In order to introduce our main results, we shall introduce  solution spaces and the notations that we employ in this article. 
\begin{definition}
For $\a \geq 0$, $f_0 \in L^{\infty}( \Gamma^- )$ and $f \in L^{\infty}(\O \times \R^3)$, we say that

\begin{enumerate}
\item    $f_0 \in L^\infty_\a(\Gamma^-)$ if 
\[
\| f_0 \|_{L^\infty_\a(\Gamma^-)} := \sup_{(X, \z) \in \Gamma^-} |f_0(X, \z)| e^{\a |\z|^2} < \infty;
\]
\item  $f_0 \in \cB_\a$ if
\[
\| f_0 \|_{\cB_\a} := \| f_0 \|_{L^\infty_\a(\Gamma^-)} + \sup_{\substack{(X, \z), (Y, \z) \in \Gamma^- \\ X \neq Y}} \frac{|f_0(X, \zeta) - f_0(Y, \zeta)|}{|X - Y|} e^{\a |\z|^2} < \infty;
\]
\item   $f_0 \in {\tilde{\cB}_\a}$ if
\[
\| f_0 \|_{\tilde{\cB}_\a} := \| f_0 \|_{\cB_\a} + \sup_{(X, \z) \in \Gamma_-} |\nabla_\z f_0(X, \z)| e^{\a |\z|^2} < \infty;
\]

\item $f \in L^\infty_\a(\O \times \R^3)$ if
\[
\| f \|_{L^\infty_\a(\O \times \R^3)} := \sup_{(x, \z) \in \O \times \R^3} |f(x, \z)| e^{\a |\z|^2} < \infty;
\]
\item $f \in \cW_\a$ if
\[
\| f \|_{\cW_\a} := \| f \|_{L^\infty_\a(\O \times \R^3)} + \sup_{(x, \z) \in \O \times \R^3} |\nabla_x f(x, \z)| w(x, \z) e^{\a |\z|^2} < \infty;
\]
\item   $f \in {\tilde{\cW}_\a}$  if
\[
\| f\|_{\tilde{\cW}_\a} := \| f \|_{\cW_\a} + \sup_{(x, \z) \in  \O \times \R^3 } |\nabla_\z f(x, \z)| w(x, \z) e^{\a |\z|^2}< \infty,
\]
 \end{enumerate}
where
\begin{align}
w(x, \z) :=& \frac{|\z|}{1 + |\z|} N(x, \z), \label{def:w}\\
N(x, \z) :=& \frac{|n(q(x, \z)) \cdot \z|}{|\z|}, \label{def:N}\\
\tau_-(x, \z) :=& \inf \{ s > 0 \mid x - s\z \notin \Omega\}, \label{def:t}\\
q(x, \z) :=& x - \tau_-(x, \z) \z \label{def:q}
\end{align}
for $(x, \z) \in \O \times (\R^3 \setminus \{ 0 \})$.
\end{definition}

\begin{definition}
Suppose $\a \geq 0$ and $0 < \sigma \leq 1$. We say that
\begin{enumerate}
\item  $\phi \in L^\infty_{\a, 1}(\O \times \R^3)$ if
\[
\| \phi \|_{L^\infty_{\a, 1}(\O \times \R^3)} := \sup_{(x, \z) \in \O \times \R^3} \frac{|\phi(x, \z)|}{1 + |\z|} e^{\a |\z|^2} < \infty;
\]
 \item $\phi \in \cC^\sigma_\a$ if 
\begin{align*}
\| \phi \|_{\cC^\sigma_\a} := \| \phi \|_{L^\infty_{\a, 1}(\O \times \R^3)} +  |\phi|_{\cC^\sigma_\a} < \infty, 
\end{align*}
where
\[
|\phi|_{\cC^\sigma_\a} := \sup_{\substack{(x, \z), (y, \z) \in \O \times (\R^3 \setminus \{0\}) \\ x \neq y}} \frac{|\phi(x, \z) - \phi(y, \z)|}{(d_{x, y}^{-\frac{1}{2}} + w_\sigma(x, y, \z)^{-1})(1 + |\z|) |x - y|^\sigma} e^{\a |\z|^2};
\]
\item $ \phi \in  \tilde{\cC}^\sigma_\a$ if
\[
\| \phi \|_{\tilde{\cC}^\sigma_\a} := \| \phi \|_{\cC^\sigma_\a} + \sup_{(x, \z) \in \O \times (\R^3 \setminus \{0\})} \frac{|\nabla_\z \phi(x, \z)|}{\left( d_x^{-\frac{1}{2}}{ + w(x, \z)^{-1}} \right) (1 + |\z|)} e^{\a |\z|^2} < \infty,
\]
where
\begin{align}
d_{x, y} :=& \min \{ d_x, d_y \},\\
d_x :=& \dist(x, \partial \Omega),\\
w_\sigma(x, y, \z) :=& \frac{|\z|^\sigma}{1 + |\z|} N(x, y, \z),\\
N(x, y, \z) :=& \min \left\{ N(x, \z), N(y, \z) \right\} 
\label{def:N2}\end{align}
for $(x, y, \z) \in {\Omega} \times {\Omega} \times (\R^3 \setminus \{ 0 \})$.
\end{enumerate}

\end{definition}

The main result of this article is the following existence theorem of differentiable solutions for the incoming boundary value problem \eqref{BVP_non_red}.

\begin{theorem} \label{main theorem 2}
Let $\O$ be an open bounded convex domain with $C^2$ boundary of positive Gaussian curvature, and let $0 \leq \a < 1/2$. Suppose \eqref{assumption_B1} holds. Then the following assertions hold.
\begin{enumerate}
\item There exists $\delta_0> 0$ such that, if $\| f_0 \|_{\tilde{\cB}_\a} < \delta_0$, then the boundary value problem \eqref{BVP_non_red} has a solution $f \in \tilde{\cW}_\a$. 
\item Moreover, in case  $0<\a <1/2$, we have $f \in W^{1, p}(\O \times \R^3)$ for all $1 \leq p < 3$.
 \end{enumerate}
\end{theorem}

\begin{remark}
Compared with a related result for diffuse reflection boundary condition in \cite{ChenKimGra}, here we have regularity for both space and velocity variables, while only regularity in space variable was discussed in \cite{ChenKimGra}. In addition, in Theorem \ref{main theorem 2} we require $C^2$ regularity for the boundary, while $C^3$ regularity was assumed in \cite{ChenKimGra}.

\end{remark}

\begin{remark}
 When the domain is small, thanks  to the contraction mapping, the existence results can be established, see \cite{CCHKS,CHKS1,CHKS2}.  Though the method is simpler, the explicit construction  gives sharp characterization to the geometric effect on the solution space for the linear problem \cite{CHKS1}. Also, to remove the condition of small domain for  nonlinear problem \cite{CHKS2} motivates the presented research.
\end{remark}
As the preparation for the proof of the nonlinear problem \eqref{BVP_non_red},  we study the existence and regularity of the  inhomogeneous linearized problem of \eqref{BVP_non_red}:
\begin{equation} \label{BVP_lin}
\begin{cases}
\z \cdot \nabla_x f = Lf + \phi &\mbox{ in } \Omega \times \mathbb{R}^{3},\\
f = f_0 &\mbox{ on } \Gamma^-.
\end{cases}
\end{equation}
The  convexity of the domain $\Omega$ allows us to rewrite the problem \eqref{BVP_lin} as the following integral form:
\begin{equation} \label{IE}
f = Jf_0 + S_\Omega \phi + S_\Omega Kf, 
\end{equation}
where
\begin{align}
Jf_0(x, \z) :=& e^{-\nu(|\z|)\tau_-(x, \z)} f_0(q(x, \z), \z), \label{J}\\   
S_\O h(x, \z) :=& \int_{0}^{\tau_-(x, \z)}e^{-\nu(|\z|)s}h(x-s\z, \z)\,ds. \label{S}
\end{align}

\begin{theorem} \label{thm:well-posed_lin}
Let $\O$ be an open bounded convex domain with $C^2$ boundary of positive Gaussian curvature, and let $0 \leq \a < 1/2$. Suppose \eqref{assumption_B1} holds. Then, for any $f_0 \in L^\infty_\a(\Gamma^-)$ and $\phi \in L^\infty_{\a, 1}(\O \times \R^3)$, there exists a unique solution $f$ to the integral equation \eqref{IE} in $L^\infty_\a(\O \times \R^3)$. Moreover, the solution $f$ satisfies the following estimate:
\begin{equation} \label{bound_Linfty}
\| f \|_{L^\infty_\a(\O \times \R^3)} \leq C \left( \| f_0 \|_{L^\infty_\a(\Gamma^-)} + \| \phi \|_{L^\infty_{\a, 1}(\O \times \R^3)} \right),
\end{equation}
where C is a positive constant independent of   $f_0$ and $\phi$.
\end{theorem}

To prove Theorem~\ref{thm:well-posed_lin}, we further rewrite the problem in the following operator form:
\[
(Id -  S_\Omega K )f = Jf_0 + S_\Omega \phi,
\]
where $Id$ is the identity map. To employ Theorem~\ref{Fredholm-powerCompact} to prove Theorem~\ref{thm:well-posed_lin}, it suffices to show that
\begin{enumerate}
\item[(A)]  $S_\Omega K:  L^\infty_\a(\O \times \R^3) \to L^\infty_\a(\O \times \R^3)$ is a power compact operator; and 
\item[(B)] the kernel space of the operator $Id -  S_\Omega K :L^\infty_\a(\O \times \R^3) \to L^\infty_\a(\O \times \R^3)$ is trivial.
\end{enumerate}
The proof of Assertion (A) can be summarized as follows: In Section 3, we show that the operator $KS_\Omega K$ sends the $L^\infty_\a(\O \times \R^3)$ functions to H\"older continuous functions. In Section 5, we set 
\[
T_R := \chi_R(\z)  K S_\O K,
\]
where $\chi_R(\z) =\chi_{\overline{B}_R(0)}(\z)$ is the characteristic function of the closed ball $\overline{B}_R(0)(\z)$. 
By applying   the Ascoli-Arzel\'a theorem on the domain $\overline{\O \times B_R}$, we see that  $T_R$ is a compact operator in   $L^\infty_\a(\O \times \R^3)$. By employing the fact that the operator norm of  $(Id - \chi_R(\z)) KS_\O K : L^\infty_\a(\O \times \R^3) \to L^\infty_\a(\O \times \R^3)$ satisfies
\[ 
\| (Id - \chi_R(\z)) KS_\O K \| \leq \frac{C}{1 + R},
\]
we see that  $T_{R} \to KS_\O K$ as $R \to \infty$. Hence,  $KS_\Omega K$ is a compact operator in  $L^\infty_\a(\O \times \R^3)$. On the other hand,  the operator $S_\Omega$ is a bounded linear operator in $L^\infty_\a(\O \times \R^3)$. We then conclude $ S_\Omega K S_\Omega K$ is a compact operator in $L^\infty_\a(\O \times \R^3)$.
The proof of Assertion (B) is based on 
coercivity estimate, see \cite{Glassey, Grad, Grad1, Mo},   
\[
- \int_{\O \times \R^3} (Lf)f\,dxd\z \geq c_0 \| (Id - P) f \|_{L^2(\O \times \R^3)}^2, \quad c_0>0,
\]
where $P$ is the projection operator from $L^2(\O \times \R^3)$ to $\mathcal{N}(L)$,  the kernel of the linear operator $L$. The detailed proof is addressed in Section 5.

In case the incoming boundary data $f_0 \in \tilde{\cB}_\a$ and the source term  $\phi \in \tilde{\cC}^1_\a$, we have the following regularity property of the solutions of the linear problem  \eqref{IE}.
\begin{theorem} \label{lem:sol_dz_lin}
Let $\O$ be an open bounded convex domain with $C^2$ boundary of positive Gaussian curvature, and let $0 \leq \a < 1/2$. Suppose \eqref{assumption_B1} holds. 
Then, for any $f_0 \in \tilde{\cB}_\a$ and $\phi \in \tilde{\cC}^1_\a$, the integral equation \eqref{IE} has a unique solution in $\tilde{\cW}_\a$ that satisfies 
\begin{equation} \label{ineq:sol_dz_lin}
\| f \|_{\tilde{\cW}_\a} \leq C \left( \| f_0 \|_{\tilde{\cB}_\a} + \| \phi \|_{\tilde{\cC}^1_\a} \right),
\end{equation}
where C is a positive constant independent of   $f_0$ and $\phi$.
\end{theorem}
\noindent The proof of Theorem~\ref{lem:sol_dz_lin} is addressed in Section 6. 

The second iteration of   the integral equation \eqref{IE} gives
\[
f(x, \z) = I(x, \z) + II(x, \z) + III(x, \z),
\]
where
\begin{align}
I(x, \z) :=& Jf_0(x, \z) + S_\Omega \phi(x, \z), \label{I}\\
II(x, \z) :=& S_\Omega K(Jf_0 + S_\Omega \phi)(x, \z) = S_\Omega K I(x, \z), \label{II}\\
III(x, \z) :=& (S_\Omega K)^2 f(x, \z). \label{III}
\end{align}
Under the assumption of Theorem~\ref{lem:sol_dz_lin}, the differentiability of  $I(x, \z)$, $II(x, \z)$ and $III(x, \z)$ are obtained in Section 6.

 To prove Theorem \ref{main theorem 2}, for $f_0 \in \tilde{\cB}_\a$,
we consider the following iteration scheme:
\begin{equation} \label{iteration_ini}
\begin{cases}
\z \cdot \nabla_x f_{1} = L f_1 &\mbox{ in } \Omega \times \R^3,\\
f_1 = f_0 &\mbox{ on } \Gamma^-, 
\end{cases}
\end{equation}
and
\begin{equation} \label{iteration}
\begin{cases}
\z \cdot \nabla_x f_{i+1} = L f_{i + 1} + \Gamma(f_{i},f_{i}) &\mbox{ in } \Omega \times \R^3,\\
f_{i + 1} = f_0 &\mbox{ on } \Gamma^-,
\end{cases}
\end{equation}
for $i \geq 1$.  In order to make use of Theorem~\ref{lem:sol_dz_lin} to prove Theorem~\ref{main theorem 2}, we prove the following lemma in Section 7.
\begin{lemma}
\label{lem88}
Let $\O$ be an open bounded convex domain with $C^2$ boundary of positive Gaussian curvature. Suppose $0 \leq \a < 1/2$ 
and \eqref{assumption_B1} holds. Then, for $h_1, h_2 \in \tilde{\cW}_\a$, we have
\begin{align}
\label{lem91}&|\Gamma(h_1, h_2)(x, \z)| \leq C \| h_1 \|_{L^\infty_\a(\O \times \R^3)}   \| h_2 \|_{L^\infty_\a(\O \times \R^3)} e^{-\alpha |\z|^2} (1 + |\z|)^\gamma,\\
\label{lem92}& |\nabla_x \Gamma(h_1, h_2)(x, \z)| \leq C { \left(d_x^{-\frac{1}{2}} + w(x, \z)^{-1} \right)} \| h_1 \|_{\tilde{\cW}_\a} \| h_2 \|_{\tilde{\cW}_\a } e^{-\alpha |\z|^2} (1 + |\z|)^\gamma,\\
\label{lem93}& |\nabla_\z \Gamma(h_1, h_2)(x, \z)| \leq C { \left( d_{x}^{-\frac{1}{2}} + w(x, \z)^{-1} \right) \| h_1 \|_{\tilde{\cW}_\a} \| h_2 \|_{\tilde{\cW}_\a} e^{-\alpha |\z|^2} (1 + |\z|)^\gamma}
\end{align}
for a.e. $(x, \z) \in \Omega \times \R^3$, where $C$ is some positive constant independent of $h_1$ and $h_2$.
\end{lemma}
Choosing $\phi = 0$ and applying Theorem~\ref{lem:sol_dz_lin}  to \eqref{iteration_ini}, we see the solution of \eqref{iteration_ini} satisfies 
\[
\| f_1 \|_{\tilde{\cW}_\a} \leq C_\sigma \| f_0 \|_{\tilde{\cB}_\a}. 
\]
Since $0 \le \gamma \le 1$, Lemma~\ref{lem88} implies  
$\|\Gamma(f_{1},f_{1})\|_{\tilde{\cC}^1_\a} \le C \| f_1 \|_{\tilde{\cW}_\a}^2 $. 
Applying Theorem~\ref{lem:sol_dz_lin}  to \eqref{iteration} by choosing $i=1$ and $\phi = \Gamma(f_{1},f_{1})$, together with Lemma~\ref{lem88}, we see that 
\[
\| f_{2} \|_{\tilde{\cW}_\a} \leq C \left( \| f_0 \|_{\tilde{\cB}_\a} + \| f_1 \|_{\tilde{\cW}_\a}^2 \right).
\]
Inductively, we obtain
\begin{equation} \label{iteration130}
\| f_{i+1} \|_{\tilde{\cW}_\a} \leq C \left( \| f_0 \|_{\tilde{\cB}_\a} + \| f_i \|_{\tilde{\cW}_\a}^2 \right).
\end{equation}
Without loss of generality, we  may assume that the constant $C$ in \eqref{iteration130} is strictly greater than $1$. We further take $0<\delta_0 < 1/(8C^2)$ so that, for  $\| f_0 \|_{\tilde{\cB}_\a} < \delta_0$, there holds
\[
\| f_1 \|_{\tilde{\cW}_\a} \leq C \times \frac{1}{8C^2} \leq \frac{1}{4C}.
\] 
Inductively, if $\| f_i \|_{\tilde{\cW}_\a} \leq 1/4C$, we see from  \eqref{iteration130} that
\begin{align*}
\| f_{i+1} \|_{\tilde{\cW}_\a} \leq  C \times \left( \frac{1}{8C^2}  +(\frac{1}{4C})^2 \right) \leq \frac{1}{4C}.
\end{align*}
Hence, the sequence $\{ f_i \}$ is uniformly bounded in $\tilde{\cW}_\a$. Furthermore, the subtraction of consecutive two equations of  \eqref{iteration} gives 
\begin{equation*}
\begin{cases}
\z \cdot \nabla_x(f_{i+1}-f_{i}) = L(f_{i+1}-f_{i})+\Gamma(f_{i},f_{i})-\Gamma(f_{i-1},f_{i-1}) &\mbox{ in } \Omega \times \R^3,\\
f_{i+1}(x,\z)-f_i(x,\z)=0 &\mbox{ on } \Gamma^-.
 \end{cases}
 \end{equation*}
Notice that $\Gamma(f_{i},f_{i})-\Gamma(f_{i-1},f_{i-1})=\Gamma(f_{i},f_{i}-f_{i-1})+\Gamma(f_{i}-f_{i-1},f_{i-1})$. Hence,
by Lemma~\ref{lem88}, we have
\begin{align*}
\| f_{i+1}-f_{i} \|_{\tilde{\cW}_\a} \leq& C ( \| f_i \|_{\tilde{\cW}_\a} + \| f_{i-1} \|_{\tilde{\cW}_\a}) \| f_{i}-f_{i-1} \|_{\tilde{\cW}_\a}\\ 
\leq& \frac{1}{2} \| f_{i}-f_{i-1} \|_{\tilde{\cW}_\a}.
\end{align*}
In the last line, we use the uniform bound $1/4C$ of $\| f_{i+1} \|_{\tilde{\cW}_\a}$. This implies the convergence of the iteration scheme \eqref{iteration} in $\tilde{\cW}_\a$ if $\| f_0 \|_{\tilde{\cB}_\a} < \delta_0$, and that the limit of $f_i$ in $\tilde{\cW}_\a$, denoted by $f$, solves the incoming boundary value problem \eqref{BVP_non_red}. This proves the first assertion of Theorem~\ref{main theorem 2}. The proof of the second assertion of Theorem~\ref{main theorem 2} is addressed in Section 8.

\section{Preliminaries} \label{sec:pre}

In this section, we shall explore the estimates and the regularizing effects of the operators  $J$, $S_\O$, and $K$ defined in \eqref{J}, \eqref{S} and \eqref{K} respectively.  We first recapitulate  some well-known estimates for the linearized collision operator $L$. In the decomposition
\[
Lh (x, \z) = - \nu(|\z|) h(x, \z) + Kh (x, \z), 
\]
where 
\[
Kh (x, \z):=\int_{\mathbb{R}^{3}}k(\z, \z_*) h(x, \z_*)\,d\z_*,
\]
as have been addressed by many authors, for any fixed  $0 < \delta < 1$ and $0 \le \gamma \le 1$, we see that $\nu$ and $k$ satisfy
\begin{align}
&k(\z, \z_*) = k(\z_*, \z), \label{AA}\\
&\nu_{0}(1+|\z|)^{\gamma} \leq \nu(|\z|) \leq \nu_{1} (1+|\z|)^{\gamma}, \label{AB}\\
&| k(\z, \z_*) | \leq C_{\delta, \gamma} \frac{1}{| \z - \z_* |(1+|\z| +| \z_* |)^{1-\gamma}} E_\delta(\z, \z_*), \label{AC}\\
&| \nabla_\z k(\z, \z_*) | \leq C_{\delta, \gamma} \frac{1+|\z|}{| \z - \z_* |^{2}(1+|\z| +| \z_* |)^{1-\gamma}}E_\delta(\z, \z_*), \label{AD}\\
&|\nabla_{\z} \nu(|\z|) | \leq C_\gamma (1+|\z|)^{\gamma-1}, \label{AE}
\end{align}
where $0 < \nu_0 \leq \nu_1$ and
\[
E_\delta(\z, \z_*) := e^{-\frac{1-\delta}{4} \left( | \z - \z_* |^{2} + \left( \frac{|\z|^{2}-| \z_* |^{2}}{| \z - \z_* |} \right)^{2} \right)}.
\]
Here and in what follows, $C$ denotes a general positive constant and it may change line by line. Also, we put a subscript to $C$ in order to emphasize that the constant $C$ depends on it. For example, $C_\alpha$ is a constant which depends on the parameter $\alpha$. Instead of making use of \eqref{AB}--\eqref{AE}, we shall employ the following estimates in the proof of our theorems:
\begin{align}
&\nu_{0} \leq \nu(|\z|) \leq \nu_{1} (1+|\z|), \label{AB'}\\
&| k(\z, \z_*) | \leq C_\delta \frac{1}{| \z - \z_* |} E_\delta(\z, \z_*), \label{AC'}\\
&| \nabla_{\z} k(\z, \z_*) | \leq C_\delta \frac{1+|\z|}{| \z - \z_* |^{2}}E_\delta(\z, \z_*), \label{AD'}\\
&| \nabla_{\z_*} k(\z, \z_*) | \leq C_\delta \frac{1+|\z_*|}{| \z - \z_* |^{2}}E_\delta(\z, \z_*), \label{AE'}\\
&| \nabla_{\z} \nu(|\z|)|\leq C \label{AF'}.
\end{align}
We remark that inequalities \eqref{AB'}--\eqref{AF'} can be derived from \eqref{AB}--\eqref{AE} under the assumption of \eqref{assumption_B1}  in the case where $0 \leq \gamma \leq 1$.

\subsection{Geometric estimates}
We first recapitulate the H\"older estimates of the functions $\tau_-(x, \z)$ and $q(x, \z)$ defined in \eqref{def:t} and \eqref{def:q} respectively.

\begin{lemma}[\cite{CHK}, Proposition 6.2] \label{lem:EPT}
Let $\Omega$ be an open bounded convex domain in $\R^3$ with $C^1$ boundary. For $x, y \in \Omega$ and $\z \in \R^3 \setminus \{ 0 \}$, we have
\begin{align*}
&|q(x, \z) - q(y, \z)| \leq \frac{|x - y|}{N(x, y, \z)},\\
&|\tau_-(x, \z) - \tau_-(y, \z)| \leq \frac{2 |x - y| }{N(x, y, \z) |\z| },
\end{align*}
where $N(x, y, \z )$ is the function defined by \eqref{def:N2}.
\end{lemma}
The following  corollary is a direct result of Lemma~\ref{lem:EPT}.

\begin{corollary} \label{cor:estTP}
Let $\Omega$ be an open bounded convex domain in $\R^3$ with $C^1$ boundary. Then, for  $x \in \Omega$ and $\z \in \R^3 \setminus \{ 0 \}$,  we have
\begin{align*}
\left| \nabla_x q(x, \z) \right| \leq& \frac{1}{N(x, \z)} = \frac{|\z|}{1 + |\z|} w(x, \z)^{-1},\\
\left| \nabla_x \tau_-(x, \z) \right| \leq& \frac{2}{N(x, \z)|\z|} = \frac{2}{1 + |\z|} w(x, \z)^{-1},
\end{align*}
where $N(x, \z)$ and $w(x, \z)$ are as defined in \eqref{def:N} and \eqref{def:w} respectively.
\end{corollary}

\begin{remark} 
The following formula of $\nabla_x \tau_-$ is known in \cite{Guo}:
\[
\nabla_x \tau_-(x, \z) = \frac{1}{N(x, \z) |\z|} n(q(x, \z))
\]
for all $(x, \z) \in \O \times (\R^3 \setminus \{ 0 \})$.
\end{remark}

For the $\z$ derivative, see \cite{CHKS1, Guo}, we know that
\begin{equation}
\label{zDerivative}
|\nabla_\z \tau_-(x, \z)| \leq \   \frac{ \tau_-(x, \z)}{N(x, \z) |\z|}
\end{equation}
for all $(x, \z) \in \O \times (\R^3 \setminus \{ 0 \})$.  Inferring from \eqref{zDerivative} and Corollary \ref{cor:estTP} and noting that
\[
\frac{\partial q_j(x, \z)}{\partial \z_i} = - \delta_{ij}\tau_-(x, \z) - \z_j \frac{ \partial \tau_-(x, \z)}{\partial \z_i},
\]
we obtain the following estimates.

\begin{lemma} \label{lem:estTPz}
Let $\Omega$ be an open bounded convex domain in $\R^3$ with $C^1$ boundary. Then, we have
\begin{align*}
\left| \nabla_\z \tau_-(x, \z) \right| \leq& \frac{ \tau_-(x, \z)}{N(x, \z)|\z|} = \frac{\tau_-(x, \z)}{1 + |\z|} w(x, \z)^{-1},\\
\left| \nabla_\z q(x, \z) \right| \leq& \frac{C|\z| \tau_-(x, \z)}{1 + |\z|} w(x, \z)^{-1} \leq \frac{C}{1 + |\z|} w(x, \z)^{-1}.
\end{align*}
\end{lemma}

\begin{proposition}[\cite{CHKS2}, Proposition 2.2] \label{prop:N1/2}
Let $\O$ be an open bounded convex domain with $C^2$ boundary of positive Gaussian curvature. Then, we have
\[
d_x^{\frac{1}{2}} \leq C N(x, \z)
\]
for all $(x, \z) \in \O \times (\R^3 \setminus \{ 0 \})$. 
\end{proposition}

For fixed  $\z \in \mathbb{R}^3 \setminus \{0 \}$, we define  
\[
\Gamma^-_\z := \{ x \in \partial \Omega \mid n(x) \cdot \z < 0 \}.
\]

\begin{proposition}[\cite{CHKS2}, Proposition 2.1] \label{prop:est_t+}
Let $\Omega$ be an open bounded convex domain in $\R^3$ with with $C^2$ boundary of positive Gaussian curvature. Then, we have
\[
|x - q(x, -\z)|\leq C N(x, \z)
\]
for all $\z \in \mathbb{R}^3 \setminus \{0 \}$ and $x \in \Gamma^-_\z$.
\end{proposition}

The following change of variables formula will be used later.

\begin{lemma}[\cite{CS}, Lemma 2.1] \label{changeintegration}
For $f \in L^1(\Omega \times \mathbb{R}^3)$, we have
\[
\int_\Omega \int_{\mathbb{R}^3} |f(x, \z)|\,d\z dx = \int_{\Gamma_-} \int_0^{\tau_+(y, \z)} |f(y + t\z, \z)| |n(y) \cdot \z|\,dt d\z d\sigma_y,
\]
where $\tau_+(y, \z) := \tau_-(y, -\z)$ and $d\sigma_y$ is the surface element on $\partial \Omega$.
\end{lemma}

The following integral estimates based on the convexity of $\O$ will be used  frequently.

\begin{proposition}[\cite{CCHS}, Lemma 5.12] \label{prop:d1/2}
Let $\Omega$ be an open bounded convex domain in $\R^3$ with $C^2$ boundary of positive Gaussian curvature. Then, we have
\[
\int_{0}^{|x - q(x, \hat{\z})|} \frac{1}{d_{x-t\hat{\z}}^{\frac{1}{2} - \epsilon}}\,dt \leq C_\epsilon
\]
for all $\epsilon \geq 0$, $x \in \overline{\Omega}$ and $\hat{\z} \in S^2$, and
\begin{equation} \label{est:1_to_1/2}
\int_{0}^{|x - q(x, \hat{\z})|} \frac{1}{d_{x-t\hat{\z}}^{1 - \epsilon}}\,dt \leq \frac{C}{\epsilon} d_x^{-\frac{1}{2} + \epsilon}
\end{equation}
for all $0 < \epsilon \leq 1/2$, $x \in \overline{\Omega}$ and $\hat{\z} \in S^2$.
\end{proposition}

\begin{lemma} \label{lem:S1/2}
Let $\Omega$ be an open bounded convex domain in $\R^3$ with $C^2$ boundary of positive Gaussian curvature.  Assume $1/2 < \sigma \leq 1$. Then, we have
\[
\int_0^{\tau_-(x, \z)} e^{-\nu(|\z|) t} \frac{1}{d_{x - t \z}^{\frac{1}{2} }}\,dt \leq \frac{C_\sigma}{|\z|^\sigma} d_x^{\sigma - 1}
\]
for all $(x, \z) \in \Omega \times (\R^3 \setminus \{ 0 \})$. 
\end{lemma}

\begin{proof}
By the change of variable  $t = r/|\z|$, we have
\begin{align*}
\int_0^{\tau_-(x, \z)} e^{-\nu(|\z|) t} \frac{1}{d_{x - t \z}^{\frac{1}{2} }}\,dt =& \frac{1}{|\z|} \int_0^{|x - q(x, \hat{\z})|} e^{-\frac{\nu(|\z|)}{|\z|}r} \frac{1}{d_{x - r\hat{\z}}^{\frac{1}{2}}}\,dr\\
\leq& \frac{1}{|\z|} \int_0^{|x - q(x, \hat{\z})|} e^{-\frac{\nu_0}{|\z|}r} \frac{1}{d_{x - r\hat{\z}}^{\frac{1}{2}}}\,dr,
\end{align*}
where $ \hat{\z} = \z / |\z|$.
For $1/2 < \sigma \leq 1$, noticing that
\[
e^{-\frac{\nu_0}{|\z|}r} \leq \frac{C_\sigma |\z|^{1-  \sigma}}{r^{1 - \sigma}},
\] 
we have
\begin{align*}
\frac{1}{|\z|} \int_0^{|x - q(x, \hat{\z})|} e^{-\frac{\nu_0}{|\z|}r} \frac{1}{d_{x - r\hat{\z}}^{\frac{1}{2}}}\,dr \leq& \frac{C_\sigma}{|\z|^\sigma} \int_0^{|x - q(x, \hat{\z})|} \frac{1}{r^{1 - \sigma} d_{x - r\hat{\z}}^{\frac{1}{2}}}\,dr\\
\leq& \frac{C_\sigma}{|\z|^\sigma} \int_0^{|x - q(x, \hat{\z})|} ( r^{\sigma - \frac{3}{2}} + d_{x - r\hat{\z}}^{\sigma - \frac{3}{2}})\,dr\\
\leq& \frac{C_\sigma}{|\z|^\sigma} d_x^{\sigma - 1}.
\end{align*}
Here, we used the estimate \eqref{est:1_to_1/2} with $\epsilon = \sigma - 1/2$ and the fact $$ d_x^{\sigma - 1} \ge |x - q(x, \hat{\z})|^{\sigma -1}. $$ This completes the proof.
\end{proof}

\begin{lemma} \label{lem:exp_inter}
Let $\O$ be an open bounded convex domain in $\R^3$ with $C^1$ boundary. Then, for $0 \leq \sigma \leq 1$, we have
\[
\left| \int_{\tau_-(x, \z)}^{\tau_-(y, \z)} e^{- \nu(|\z|) t}\,dt \right| \leq  C_\sigma \frac{ |x - y|^\sigma  }{|\z|^\sigma N(x, y, \z)}.  
\]
for all $x, y \in \O$ and $\z \in \R^3 \setminus \{ 0 \}$.
\end{lemma}

\begin{proof} Without loss of generality, we may assume that $\tau_-(x, \z) \leq \tau_-(y, \z)$. By the assumption \eqref{AB'}, we have
\begin{align*}
\int_{\tau_-(x, \z)}^{\tau_-(y, \z)} e^{-\nu(|\z|)t}\,dt \leq& \left( \int_0^\infty e^{- \nu_0 t}\,dt \right)^{1 - \sigma} \left( \int_{\tau_-(x, \z)}^{\tau_-(y, \z)} e^{- \nu_0 t}\,dt \right)^{\sigma}\\
\leq& C_\sigma |\tau_-(y, \z) - \tau_-(x, \z)|^\sigma
\end{align*}
for all $0 \leq \sigma \leq 1$, $x, y \in \O$ and $\z \in \R^3 \setminus \{ 0 \}$. Applying Lemma \ref{lem:EPT}, we  obtain
\[
|\tau_-(y, \z) - \tau_-(x, \z)|^\sigma \leq C_\sigma  \left( \frac{|x - y|}{|\z|  N(x, y, \z) } \right)^\sigma \leq C_\sigma \frac{|x- y|^\sigma}{|\z|^\sigma N(x, y, \z)}.
\]
This completes the proof.
\end{proof}

\begin{proposition}[\cite{CHKS1}, Corollary 2.3] \label{prop:est_S_convex}
Let $\O$ be an open bounded convex domain in $\R^3$. Then, for $\z \neq 0$, $a > 0$ and $b \geq 0$, we have
\[
\int_0^{\tau_-(x, \z)} t^b e^{-a t}\,dt \leq C_{a, b} \min \left\{ 1, \frac{1}{|\z|} \right\}
\]
for all $(x, \z) \in \Omega \times (\R^3 \setminus \{ 0 \})$.
\end{proposition}

The estimate in Proposition \ref{prop:est_S_convex} can be expressed as
\begin{equation} \label{est:S_decay}
\int_0^{\tau_-(x, \z)} t^b e^{-a t}\,dt \leq \frac{C_{a, b}}{1 + |\z|}
\end{equation}
for all $(x, \z) \in \Omega \times (\R^3 \setminus \{ 0 \})$. 

\begin{proposition} \label{prop:dist_interpolate}
Let $\O$ be an open bounded convex domain in $\R^3$ and $x, y \in \O$. Set $x(t) := x + (y - x)t$ for $0 \leq t \leq 1$. Then, we have $d_{x(t)} \geq d_{x, y}$ for all $0 \leq t \leq 1$.
\end{proposition}

\begin{proof}
By the definition of $d_{x, y}$, the two balls $B_x := B(x, d_{x, y})$ and $B_y := B(y, d_{x, y})$ are contained in $\O$. Moreover, since $\O$ is convex, the convex hull of $B_x$ and $B_y$, denoted by $\conv(B_x, B_y)$, is also contained in $\O$. Thus, we have
\[
d_{x(t)} = d(x(t), \partial \O) \geq d(x(t), \partial \conv (B_x, B_y)) = d_{x, y}.
\]
This completes the proof.
\end{proof}

\begin{proposition} \label{prop:N_interpolation}
Let $\O$ be an open bounded convex domain in $\R^3$ and  $x, y \in \O$. Set $x(t) := x + t(y - x)$ for $0 \leq t \leq 1$. Then, there holds
\[
N(x(t), \z) \geq  N(x, y, \z)
\] 
for all $0 \leq t \leq 1$ and $\z \in \R^3 \setminus \{ 0 \}$.
\end{proposition}

\begin{proof}
The estimate obviously holds true  if the straight line $\overleftrightarrow{xy}$ is parallel to the vector $\z$. Thus, we only need to consider the case where the straight line $\overleftrightarrow{xy}$ is not parallel to the vector $\z$. In this latter case, we consider the plane which  contains the straight line $\overleftrightarrow{xy}$ and the vector $\z$.  We now proceed our argument on this plane.

We measure the signed angle $\theta(t) \in (-\pi/2, \pi/2)$ between $n(q(x(t), \z))$ and $-\z$ in such a way that it is positive when the alignment of the vector pair $\{-\z$, $n(q(x(t), \z))\}$ is counterclockwise. Then, we see that $N(x(t), \z) = \cos \theta(t)$. Due to the convexity of $\O$, we see that $\theta(t)$ is a monotonic function of $t$. Thus, we obtain
\[
N(x(t), \z) \geq \cos \left( \max_{0 \leq t \leq 1} |\theta(t)| \right) = \min \{ \cos \theta(0), \cos \theta(1) \} = \min \{ N(x, \z), N(y, \z) \}.
\]
This completes the proof.
\end{proof}

\subsection{The function space $\cC^\sigma_\a$}

In this subsection, we mention two properties for the function space $\cC^\sigma_\a$.

\begin{proposition} \label{prop:Csa_order}
Let $0 < \sigma_1 \leq \sigma_2 \leq 1$ and $\a \geq 0$. Then, we have $\cC^{\sigma_2}_\a \subset \cC^{\sigma_1}_\a$. Moreover, we have 
\[
\| \phi \|_{\cC^{\sigma_1}_\a} \leq C \| \phi \|_{\cC^{\sigma_2}_\a}
\] 
for all $\phi \in \cC^{\sigma_2}_\a$.
\end{proposition}

\begin{proof}
Let $\phi \in \cC^{\sigma_2}_\a$. Then, we have
\[
| \phi(x, \z) | \leq \| \phi \|_{L^\infty_{\a, 1}(\O \times \R^3)} (1 + |\z|) e^{-\a |\z|^2}
\]
for all $(x, \z) \in \O \times \R^3$ and
\[
|\phi(x, \z) - \phi(y, \z)| \leq | \phi |_{\cC^{\sigma_2}_\a} (1 + |\z|) \left( d_{x, y}^{-\frac{1}{2}} + w_{\sigma_2}(x, y, \z)^{-1} \right) |x - y|^{\sigma_2} e^{-\a |\z|^2}
\]
for all $(x, \z), (y, \z) \in \O \times \R^3$. Thus, we have
\begin{align*}
|\phi(x, \z) - \phi(y, \z)| =& |\phi(x, \z) - \phi(y, \z)|^{1-p} |\phi(x, \z) - \phi(y, \z)|^p\\
\leq& (2 \| \phi \|_{L^\infty_{\a, 1}(\O \times \R^3)} (1 + |\z|) e^{-\a |\z|^2})^{1 - p}\\
&\times | \phi |_{\cC^{\sigma_2}_\a}^p (1 + |\z|)^p \left( d_{x, y}^{-\frac{1}{2}} + w_{\sigma_2}(x, y, \z)^{-1} \right)^p |x - y|^{p \sigma_2} e^{-p\a |\z|^2}\\
\leq& C_p \| \phi \|_{\cC^{\sigma_2}_\a} (1 + |\z|)\\
&\times \left( d_{x, y}^{-\frac{p}{2}} + N(x, y, \z)^{-p} \frac{(1 + |\z|)^p}{|\z|^{p \sigma_2}} \right) |x - y|^{p \sigma_2} e^{-\a |\z|^2}\\
\leq& C_p \| \phi \|_{\cC^{\sigma_2}_\a} (1 + |\z|) \left( d_{x, y}^{-\frac{1}{2}} + w_{p \sigma_2}(x, y, \z)^{-1} \right) |x - y|^{p \sigma_2} e^{-\a |\z|^2}
\end{align*}
for all $0 \leq p \leq 1$. Letting $p = \sigma_1/\sigma_2$, we have
\[
|\phi(x, \z) - \phi(y, \z)| \leq C \| \phi \|_{\cC^{\sigma_2}_\a} (1 + |\z|) \left( d_{x, y}^{-\frac{1}{2}} + w_{\sigma_1}(x, y, \z)^{-1} \right) |x - y|^{\sigma_1} e^{-\a |\z|^2}.
\]
Thus, we conclude that $\phi \in C^{\sigma_1}_\a$ and 
\[
\| \phi \|_{\cC^{\sigma_1}_\a} \leq C \| \phi \|_{\cC^{\sigma_2}_\a}.
\]
This completes the proof.
\end{proof}

\begin{proposition}
Let $\phi$ be a function on $\O \times \R^3$ such that $|\phi(x, \z)| \leq C e^{- \a |\z|^2}$ for all $(x, \z) \in \O \times \R^3$ and 
\[
|\nabla_x \phi(x, \z)| \leq C \left( d_x^{-\frac{1}{2}} + w(x, \z)^{-1} \right) (1 + |\z|) e^{-\a |\z|^2}
\]
for a.e. $(x, \z) \in \O \times \R^3$. Then the function $\phi$ belongs to $\cC^1_\a$. 
\end{proposition}

\begin{proof}
Let $x(t)$ lie on the line segment $\overline{xy}$ as defined in Proposition \ref{prop:dist_interpolate}. Then, by the fundamental theorem of calculus, we have
\begin{align*}
|\phi(x, \z) - \phi(y, \z)| =& \left| \int_0^1 \frac{d}{dt} \phi(x(t), \z)\,dt \right|\\
\leq& \int_0^1 |\nabla_x \phi(x(t), \z)|\,dt |x - y|\\
\leq& C \int_0^1 \left( d_{x(t)}^{-\frac{1}{2}} + w(x(t), \z)^{-1} \right)\,dt (1 + |\z|) e^{-\a |\z^2|} |x - y|.
\end{align*}
By Proposition \ref{prop:dist_interpolate}, we have $d_{x(t)} \geq d_{x, y}$ for all $0 \leq t \leq 1$. Also, by Proposition \ref{prop:N_interpolation}, we have $w(x(t), \z) \geq w(x, y, \z)$. Thus, we have
\begin{align*}
\int_0^1 \left( d_{x(t)}^{-\frac{1}{2}} + w(x(t), \z)^{-1} \right)\,dt \leq d_{x, y}^{-\frac{1}{2}} + w(x, y, \z)^{-1}. 
\end{align*}
Therefore, we obtain $\| \phi \|_{\cC^1_\a} \leq C$, which implies that $\phi \in \cC^1_\a$. This completes the proof. 
\end{proof}


\subsection{The estimates for $J$} \label{subsec:J}

A simple calculation shows the following boundedness of the operator $J$.

\begin{proposition} \label{prop:bound_J_Calpha}
Assume $\O$ is an open bounded convex domain and $\a \geq 0$. Then, the operator $J: L^\infty_\a(\Gamma^-) \to L^\infty_\a(\Omega \times \R^3)$ defined by \eqref{J} is bounded. 
\end{proposition}

Next, we show a H\"older estimate for $J$.

\begin{lemma} \label{lem:J_Hol}
Assume $\O$ is  an open bounded convex domain and $\a \geq 0$. Then, for $f_0 \in \cB_\a$ and $0 \leq \sigma \leq 1$, we have
\[
|Jf_0 (x, \z) - Jf_0 (y, \z)| \leq C_\sigma \| f_0 \|_{\cB_\a} N(x, y, \z)^{-1} \frac{1 + |\z|}{|\z|^\sigma} |x - y|^{\sigma} e^{-\a |\z|^2}
\]
for all $x, y \in \O$ and $\z \in \R^3 \setminus \{ 0 \}$.
\end{lemma}

\begin{proof}
Without loss of generality, we assume that $\tau_-(x, \z) \leq \tau_-(y, \z)$. By the definition of the function space $\cB_\a$, we get
\begin{align*}
|Jf_0 (x, \z) - Jf_0 (y, \z)| \leq& \| f_0 \|_{\cB_\a} |q(x, \z) - q(y, \z)| e^{-\a |\z|^2}\\ 
&+ \| f_0 \|_{L^\infty_\a(\Gamma^-)} e^{-\a |\z|^2} \left| e^{-\nu(|\z|) \tau_-(x, \z)} - e^{-\nu(|\z|) \tau_-(y, \z)} \right|.
\end{align*}
For the first term on the right hand side, we apply Lemma \ref{lem:EPT} to obtain
\[
|q(x, \z) - q(y, \z)| \leq \frac{|x - y|}{ N(x, y, \z)} \leq C_\sigma  \frac{1 + |\z|}{|\z|^{\sigma}} \frac{ |x - y|^\sigma}{ N(x, y, \z)},
\]
where $C_{\sigma}$ has been properly chosen so that 
\[
\diam(\O)^{1 - \sigma} \le C_\sigma  \frac{1 + |\z|}{|\z|^{\sigma}}.
\]
For the second term on the right hand side, by the assumption \eqref{AB'} and  Lemma \ref{lem:exp_inter}, we have
\begin{align*}
\left| e^{-\nu(|\z|) \tau_-(x, \z)} - e^{-\nu(|\z|) \tau_-(y, \z)} \right| =& \left| \nu(|\z|) \int_{\tau_-(x, \z)}^{\tau_-(y, \z)} e^{-\nu(|\z|)t}\,dt \right|\\
\leq& \nu_1 (1 + |\z|) \left| \int_{\tau_-(x, \z)}^{\tau_-(y, \z)} e^{-\nu(|\z|)t}\,dt \right|\\
\leq& C_\sigma N(x, y, \z)^{-1} \frac{1 + |\z|}{|\z|^\sigma} |x- y|^\sigma.
\end{align*}
This completes the proof of Lemma~\ref{lem:J_Hol}.
\end{proof}

The H\"older estimate with $\sigma = 1$ gives the following bound for the derivative of $J$.

\begin{corollary} \label{cor:J_dx}
Let $\O$ be an open bounded convex domain and $\a \geq 0$. Then, for $f_0 \in \cB_\a$, we have
\[
|\nabla_x Jf_0 (x, \z)| \leq C \| f_0 \|_{\cB_\a} w(x, \z)^{-1} e^{-\a |\z|^2}
\]
for a.e. $(x, \z) \in \O \times \R^3$.
\end{corollary}

For the $\z$ derivative, we have the following estimate.

\begin{lemma} \label{lem:J_dz}
Let $\O$ be an open bounded convex domain and $\a \geq 0$. Then, for $f_0 \in \tilde{\cB}_\a$, we have
\[
|\nabla_\z Jf_0 (x, \z)| \leq C \| f_0 \|_{\tilde{\cB}_\a} w(x, \z)^{-1} e^{-\a |\z|^2}
\]
for a.e. $(x, \z) \in \O \times \R^3$.
\end{lemma}

\begin{proof}
The formal differentiation yields
\begin{align*}
\nabla_\z Jf_0(x, \z) =& -(\nabla_\z \nu(|\z|)) \tau_-(x, \z) J f_0(x, \z) - \nu(|\z|) (\nabla_\z \tau_-(x, \z)) J f_0(x, \z)\\
&+ e^{-\nu(|\z|) \tau_-(x, \z)} \left( \nabla_\z q(x, \z) \right) \nabla_X f_0 (q(x, \z), \z)\\
&+ e^{-\nu(|\z|) \tau_-(x, \z)} (\nabla_\z f_0) (q(x, \z), \z),
\end{align*}
where $\nabla_X f_0$ denotes the covariant derivative of $f_0$ on $\partial \O$. 
To prove the lemma, we estimate each term on the right hand side.
For the first term, by  \eqref{AF'} and the estimate
\[
\tau_-(x, \z) e^{- \nu(|\z|) \tau_-(x, \z)} \leq C,
\]
we have
\begin{align*}
|(\nabla_\z \nu(|\z|)) \tau_-(x, \z) J f_0(x, \z)| \leq& C|f_0(q(x, \z), \z)|\\
\leq& C \| f_0 \|_{\tilde{\cB}_\a} e^{-\a |\z|^2}\\ 
\leq& C \| f_0 \|_{\tilde{\cB}_\a} w(x, \z)^{-1} e^{-\a |\z|^2}.
\end{align*}
For the second term, by Lemma \ref{lem:estTPz}, we obtain
\begin{align*}
|\nu(|\z|) (\nabla_\z \tau_-(x, \z)) J f_0(x, \z)| \leq& \left(\sup_{t > 0} t e^{-t} \right) \frac{1}{1 + |\z|} |f_0(q(x, \z), \z)| w(x, \z)^{-1}\\ 
\leq& C \| f_0 \|_{\tilde{\cB}_\a} w(x, \z)^{-1} e^{-\a |\z|^2}. 
\end{align*}

Since $f_0 \in \tilde{\cB}_\a$,  we see that
$$
\sup_{\substack{(X, \z), (Y, \z) \in \Gamma^- \\ X \neq Y}} \frac{|f_0(X, \zeta) - f_0(Y, \zeta)|}{|X - Y|} e^{\a |\z|^2} \le  \| f_0 \|_{\tilde{\cB}_\a}, 
$$
which implies 
\[
|\nabla_X f_0(X, \z)| \leq  \| f_0 \|_{\tilde{\cB}_\a} e^{- \a |\z|^2}
\]
for a.e. $(X, \z) \in \Gamma^-$.
Applying Lemma \ref{lem:estTPz} to the third term,  we get
\[
|e^{-\nu(|\z|) \tau_-(x, \z)} \left( \nabla_\z q(x, \z) \right) \nabla_X f_0 (q(x, \z), \z)| \leq C \| f_0 \|_{\tilde{\cB}_\a} w(x, \z)^{-1} e^{-\a |\z|^2}.
\]
 For the last term, it is easy to see
\begin{align*}
|e^{-\nu(|\z|) \tau_-(x, \z)} (\nabla_\z f_0) (q(x, \z), \z)| \leq& C |(\nabla_\z f_0) (q(x, \z), \z)|\\ 
\leq& C \| f_0 \|_{\tilde{\cB}_\a} w(x, \z)^{-1} e^{-\a |\z|^2}.
\end{align*}
Therefore, the proof of the lemma is complete.

\end{proof}

\subsection{The estimates for $S_\O$} \label{subsec:S}

\begin{proposition} \label{prop:S_decay}
Let $\O$ be an open bounded convex domain and $\a \geq 0$. Then, for $h \in L^\infty_{\a, 1}(\O \times \R^3)$, we have
\[
|S_\O h(x, \z)| \leq C \| h \|_{L^\infty_{\a, 1}(\O \times \R^3)} e^{-\a |\z|^2}
\]
for a.e. $(x, \z) \in \O \times \R^3$.
\end{proposition}

\begin{proof}
For $h \in L^\infty_{\alpha, 1}(\O \times \R^3)$, we have
\begin{align*}
|S_\O h(x, \z)| e^{\alpha |\z|^2} \leq& \int_0^{\tau_-(x, \z)} e^{- \nu(|\z|) t} |h(x - t\z, \z)| e^{\alpha |\z|^2} \,dt\\
\leq& \| h \|_{L^\infty_{\alpha, 1}(\O \times \R^3)} (1 + |\z|) \int_0^{\tau_-(x, \z)} e^{- \nu_0 t} \,dt
\end{align*}
for a.e. $(x, \z) \in \O \times \R^3$. The conclusion follows from the estimate \eqref{est:S_decay} with $a = \nu_0$ and $b = 0$. This completes the proof.
\end{proof}

\begin{corollary} \label{cor:bound_S_Calpha}
Let $\O$ be an open bounded convex domain and $\a \geq 0$. Then, the operator $S_\O: L^\infty_{\a, 1}(\O \times \R^3) \to L^\infty_\a(\O \times \R^3)$ defined by \eqref{S} is bounded.
\end{corollary}

\begin{lemma} \label{lem:S_Hol}
Let $\O$ be an open bounded convex domain with $C^2$ boundary of positive Gaussian curvature. Assume $1/2 < \sigma \leq 1$ and $\a \geq 0$. Then, for $\phi \in \cC^\sigma_\a$, we have
\[
|S_\O \phi(x, \z) - S_\O \phi(y, \z)| \leq C \| \phi \|_{\cC^\sigma_\a} \left( N(x, y, \z)^{-1} + d_{x, y}^{\sigma - 1} \right) \frac{1 + |\z|}{|\z|^\sigma} |x - y|^\sigma e^{-\a |\z|^2}
\]
for all $x, y \in \O$ and $\z \in \R^3 \setminus \{ 0 \}$. 
\end{lemma}

\begin{proof}
Without loss of generality, we may assume that $\tau_-(x, \z) \leq \tau_-(y, \z)$. Hence, 
\begin{align*}
|S_\O \phi(x, \z) - S_\O \phi(y, \z)| \leq& \int_0^{\tau_-(x, \z)} e^{-\nu(|\z|) t} |\phi(x - t\z, \z) - \phi(y - t\z, \z)|\,dt\\
&+ \int_{\tau_-(x, \z)}^{\tau_-(y, \z)} e^{-\nu(|\z|) t} |\phi(y - t\z, \z)|\,dt.
\end{align*}

Applying Proposition \ref{prop:d1/2} and Lemma \ref{lem:S1/2} to the above integral,  we  obtain

\begin{align*}
&\int_0^{\tau_-(x, \z)} e^{-\nu(|\z|) t} |\phi(x - t\z, \z) - \phi(y - t\z, \z)|\,dt\\
\leq& \| \phi \|_{\cC^\sigma_\a} (1 + |\z|) |x - y|^\sigma e^{- \a |\z|^2}\\
&\times \int_0^{\tau_-(x, \z)} e^{- \nu(|\z|) t} \left( d_{x - t\z, y - t\z}^{-\frac{1}{2}} + N(x-t\z, y-t\z, \z)^{-1} \frac{1 + |\z|}{|\z|^\sigma} \right)\,dt\\
\leq& C_\sigma \| \phi \|_{\cC^\sigma_\a} \left( d_{x, y}^{\sigma - 1} + N(x, y, \z)^{-1} \right) \frac{1 + |\z|}{|\z|^\sigma} |x - y|^\sigma e^{- \a |\z|^2}.
\end{align*}
Here, we have employed \eqref{est:S_decay} and the fact that $N(x-t\z, \z) = N(x, \z)$ since $q(x -t\z, \z) = q(x, \z)$.

For the second integral on the right hand side, we use Lemma \ref{lem:exp_inter} to get 
\begin{align*}
\int_{\tau_-(x, \z)}^{\tau_-(y, \z)} e^{-\nu(|\z|) t} |\phi(y - t\z, \z)|\,dt \leq& \| \phi \|_{L^\infty_{\a, 1}(\O \times \R^3)} (1 + |\z|) e^{- \a |\z|^2} \left| \int_{\tau_-(x, \z)}^{\tau_-(y, \z)} e^{-\nu_0 t}\,dt \right|\\
\leq& C_\sigma \| \phi \|_{\cC^\sigma_\a} N(x, y, \z)^{-1} \frac{1 + |\z|}{|\z|^\sigma} |x - y|^\sigma e^{-\a |\z|^2}.
\end{align*}
This completes the proof.
\end{proof}
By choosing $\sigma =1$ in Lemma~\ref{lem:S_Hol}, we obtain the following  corollary.
\begin{corollary} \label{cor:S_dx}
Let $\O$ be an open bounded convex domain with $C^1$ boundary. Then, for $\phi \in \cC^1_\a$, we have
\[
|\nabla_x S_\O \phi(x, \z)| \leq C \| \phi \|_{\cC^1_\a} w(x, \z)^{-1} e^{-\a |\z|^2}
\]
for a.e. $(x, \z) \in \O \times \R^3$, where $w$ is the function defined by \eqref{def:w}.
\end{corollary}

For the $\z$ derivative, we have the following estimate.

\begin{lemma} \label{lem:S_dz}
Let $\O$ be an open bounded convex domain with $C^1$ boundary. Then, for $\phi \in \tilde{\cC}^1_\a$, we have
\[
|\nabla_\z S_\O \phi(x, \z)| \leq C \| \phi \|_{\tilde{\cC}^1_\a} w(x, \z)^{-1} e^{-\a |\z|^2}
\]
for a.e. $(x, \z) \in \O \times \R^3$.
\end{lemma}

\begin{proof}
We remark that a function $\phi$ belonging to $\cC^1_\a$ can be extended continuously up to $\Gamma^-$ and we have $\| \phi \|_{L^\infty_{\a, 1}(\Gamma^-)} \leq \| \phi \|_{L^\infty_{\a, 1}(\O \times \R^3)}$, where
\[
\| \phi \|_{L^\infty_{\a, 1}(\Gamma^-)} := \sup_{(X, \z) \in \Gamma^-} \frac{|\phi(X, \z)|}{1+ |\z|} e^{\a |\z|^2}.
\]
By the formal differentiation, we have
\begin{align*}
\nabla_\z S_\O \phi(x, \z) =& \nabla_\z \tau_-(x, \z) e^{-\nu(|\z|) \tau_-(x, \z)} \phi(q(x, \z), \z)\\
&- (\nabla_\z \nu(|\z|)) \int_0^{\tau_-(x, \z)} t e^{- \nu(|\z|) t} \phi(x - t\z, \z)\,dt\\
&- \int_0^{\tau_-(x, \z)} t e^{- \nu(|\z|) t} (\nabla_x \phi)(x - t\z, \z)\,dt\\
&+ \int_0^{\tau_-(x, \z)} e^{- \nu(|\z|) t} (\nabla_\z \phi)(x - t\z, \z)\,dt.
\end{align*}
For the first term on the right hand side, by Corollary \ref{cor:estTP}, we see that
\begin{align*}
&|\nabla_\z \tau_-(x, \z) e^{-\nu(|\z|) \tau_-(x, \z)} \phi(q(x, \z), \z)|\\ 
\leq& |\nabla_x \tau_-(x, \z)| \tau_-(x, \z) e^{- \nu_0 \tau_-(x, \z)} \| \phi \|_{L^\infty_{\a, 1}(\Gamma^-)} (1 + |\z|) e^{- \a |\z|^2}\\
\leq& C \| \phi \|_{\tilde{\cC}^1_\a} w(x, \z)^{-1} e^{- \a |\z|^2}.
\end{align*}
For the second term, we apply the estimate \eqref{est:S_decay} to obtain
\begin{align*}
&\left| (\nabla_\z \nu(|\z|)) \int_0^{\tau_-(x, \z)} t e^{- \nu(|\z|) t} \phi(x - t\z, \z)\,dt \right|\\
\leq& C \| \phi \|_{L^\infty_{\a, 1}(\O \times \R^3)} (1 + |\z|) e^{- \a |\z|^2} \int_0^{\tau_-(x, \z)} t e^{- \nu(|\z|) t}\,dt\\
\leq& C \| \phi \|_{\tilde{\cC}^1_\a} w(x, \z)^{-1} e^{-\a |\z|^2}.
\end{align*}
For the third term, applying Lemma \ref{lem:S1/2} with $\sigma = 1$, we get
\begin{align*}
&\left| \int_0^{\tau_-(x, \z)} t e^{- \nu(|\z|) t} (\nabla_x \phi)(x - t\z, \z)\,dt \right|\\
\leq& \| \phi \|_{\cC^1_\a} (1 + |\z|) e^{- \a |\z|^2} \int_0^{\tau_-(x, \z)} t e^{- \nu_0 t} \left( \frac{1}{d_{x - t\z}^{\frac{1}{2}}} + w(x-t \z, \z)^{-1} \right)\,dt\\
\leq& C \| \phi \|_{\tilde{\cC}^1_\a} (1 + |\z|) e^{- \a |\z|^2} \int_0^{\tau_-(x, \z)} e^{- \frac{\nu_0}{2} t} \left( \frac{1}{d_{x - t\z}^{\frac{1}{2}}} +  w(x, \z)^{-1} \right)  \,dt\\
\leq& C \| \phi \|_{\tilde{\cC}^1_\a} \left( \frac{1 + |\z|}{|\z|} +w(x, \z)^{-1} \right) e^{- \a |\z|^2}\\
\leq& C \| \phi \|_{\tilde{\cC}^1_\a} w(x, \z)^{-1} e^{- \a |\z|^2},
\end{align*}
where we have used the fact  $w(x-t\z, \z) = w(x, \z)$. Applying Lemma \ref{lem:S1/2} with $\sigma = 1$ again to the last term, in a similar calculation as we do for the third term, we obtain
\begin{equation*}
\left| \int_0^{\tau_-(x, \z)} e^{- \nu(|\z|) t} (\nabla_\z \phi)(x - t\z, \z)\,dt \right| \leq C \| \phi \|_{\tilde{\cC}^1_\a} w(x, \z)^{-1} e^{- \a |\z|^2}.
\end{equation*}
Therefore, the proof is completed.
\end{proof}

\subsection{Regularizing effect of $K$} \label{subsec:K}

We introduce some estimates for the integral operator $K$ with the integral kernel $k$, which are modifications of the ones obtained in \cite{CHKS1}.

\begin{lemma} \label{lem:est_k_0}
For $0 \leq a < 1/2$ and $0 \leq \sigma \leq 1$, we have
\[
\int_{\R^3} \frac{1 + |\z_*|}{|\z_*|^\sigma} |k(\z, \z_*)| e^{- a |\z_*|^2}\,d\z_* \leq C_a e^{- a |\z|^2}
\]
for all $\z \in \R^3$.
\end{lemma}

\begin{proof}
We observe that
\[
\frac{1 + |\z_*|}{|\z_*|^\sigma} = |\z_*|^{-\sigma} + |\z_*|^{1 - \sigma} \leq 2(|\z_*|^{-1} + |\z_*|)
\]
for all $\z_* \in \R^3$. Thus, it suffices to show that
\begin{equation} \label{est:k-1}
\int_{\R^3} \frac{1}{|\z_*|} |k(\z, \z_*)| e^{- a |\z_*|^2}\,d\z_* \leq C_a e^{- a |\z|^2}
\end{equation}
and
\begin{equation} \label{est:k1}
\int_{\R^3} |\z_*| |k(\z, \z_*)| e^{- a |\z_*|^2}\,d\z_* \leq C_a e^{- a |\z|^2}.
\end{equation}

 A direct calculation shows
\begin{equation*}
\begin{split}
&-\frac{1 - \delta}{4} \left( |\z - \z_*|^2 + \left( \frac{|\z|^2 - |\z_*|^2}{|\z - \z_*|} \right)^2 \right)\\ 
=& a |\z|^2 - \alpha_{1, a, \delta} |\z - \z_*|^2 -(1-\delta) \left( \frac{(\z - \z_*) \cdot \z}{|\z - \z_*|} - \alpha_{2, a, \delta} |\z - \z_*| \right)^2 - a |\z_*|^2\\
=& -a |\z|^2 - \alpha_{1, a, \delta} |\z - \z_*|^2 -(1-\delta) \left( \frac{(\z - \z_*) \cdot \z_*}{|\z - \z_*|} + \alpha_{2, a, \delta} |\z - \z_*| \right)^2 + a |\z_*|^2
\end{split}
\end{equation*}
for all $\z, \z_* \in \R^3$, where
\begin{align*} 
\alpha_{1, a, \delta} :=& \frac{(1 - \delta + 2a)(1 - \delta - 2a)}{4(1 -\delta)}, \\
\alpha_{2, a, \delta} :=& \frac{1 - \delta - 2a}{2(1-\delta)}.
\end{align*}
Thus, we have
\begin{equation}
\label{factorE}
E_\delta(\z, \z_*) = e^{-a |\z|^2} e^{-\a_{1, a, \delta} |\z - \z_*|^2} e^{-(1-\delta) \left( \frac{(\z - \z_*) \cdot \z_*}{|\z - \z_*|} + \alpha_{2, a, \delta} |\z - \z_*| \right)^2} e^{a |\z_*|^2}.
\end{equation}
In what follows, we set $\delta = (1 - 2a)/2$ so that we have $\a_{1, a, \delta} > 0$.

We first show the estimate \eqref{est:k-1}. We have
\begin{align*}
\int_{\R^3} \frac{1}{|\z_*|} |k(\z, \z_*)| e^{- a |\z_*|^2}\,d\z_* \leq& C_a e^{- a |\z|^2} \int_{\R^3} \frac{1}{|\z_*| |\z - \z_*|} e^{- \a_{1, a, \delta} |\z - \z_*|^2}\,d\z_*\\
\leq& C_a e^{- a |\z|^2} \left( \int_{\{ |\z_*| \leq |\z - \z_*| \}} \frac{1}{|\z_*|^2} e^{- \a_{1, a, \delta} |\z_*|^2}\,d\z_* \right.\\
&\left.  + \int_{\{ |\z_*| > |\z - \z_*| \}} \frac{1}{|\z - \z_*|^2} e^{- \a_{1, a, \delta} |\z - \z_*|^2}\,d\z_* \right)\\
\leq& C_a e^{- a |\z|^2}.
\end{align*}
This proves the estimate \eqref{est:k-1}.

We next prove the second inequality \eqref{est:k1}. By the triangle inequality and \eqref{AC'}, we have
\begin{align*}
&\int_{\R^3} |\z_*| |k(\z, \z_*)| e^{- a |\z_*|^2}\,d\z_*\\ 
&\leq  \int_{\R^3} (|\z| + |\z - \z_* | ) |k(\z, \z_*)| e^{- a |\z_*|^2}\,d\z_*\\
\leq& |\z| \int_{\R^3} |k(\z, \z_*)| e^{- a |\z_*|^2}\,d\z_* + C_a \int_{\R^3} E_\delta(\z, \z_*) e^{- a |\z_*|^2}\,d\z_*. 
\end{align*}
For the second integral on the right hand side, in the same way as the derivation of the estimate \eqref{est:k-1}, we have
\[
\int_{\R^3} E_\delta(\z, \z_*) e^{- a |\z_*|^2}\,d\z_* \leq C_a e^{- a |\z|^2}.
\]
For the first integral on the right hand side, we make use of the third factor on the right hand side of \eqref{factorE}. Namely, by introducing the spherical coordinates $\z_* = \z + r \omega$, $r > 0$, $\omega \in S^2$ and letting $t := \omega \cdot \z/|\z|$, we have 
\[
\int_{\R^3} |k(\z, \z_*)| e^{- a |\z_*|^2}\,d\z_* \leq C_a e^{- a |\z|^2 } \int_0^\infty \int_{-1}^1 r e^{-\a_{1, a, \delta} r^2} e^{-(1 - \delta) (|\z| t - \a_{2, a, \delta} r)^2}\,dtdr.
\]
For $|\z| \leq 1$, we see that
\[
\int_{-1}^1 e^{-(1 - \delta) (|\z| t - \a_{2, a, \delta} r)^2}\,dt \leq 2.
\]
On the other hand, when $|\z| > 1$, introducing a new variable $s = |\z|t - \a_{2, a, \delta} r$, it is estimated as
\begin{align*}
\int_{-1}^1 e^{-(1 - \delta) (|\z| t - \a_{2, a, \delta} r)^2}\,dt = \frac{1}{|\z|} \int_{-|\z| - \a_{2, a, \delta} r}^{|\z| - \a_{2, a, \delta} r} e^{-(1 -\delta) s^2}\,ds \leq \frac{C_\delta}{|\z|}.
\end{align*}
Thus, we have
\[
\int_{-1}^1 e^{-(1 - \delta) (|\z| t - \a_{2, a, \delta} r)^2}\,dt \leq C_\delta (1 + |\z|)^{-1}.
\]
Since the constant in the above estimate is independent of $r$, we conclude that
\[
\int_0^\infty \int_{-1}^1 r e^{-\a_{1, a, \delta} r^2} e^{-(1 - \delta) (|\z| t - \a_{2, a, \delta} r)^2}\,dtdr \leq C_a (1 + |\z|)^{-1}.
\]
Therefore, the estimate \eqref{est:k1} is proved.
\end{proof}

Following the same argument as in the proof of Lemma \ref{lem:est_k_0}, we have the following estimate.

\begin{lemma} \label{lem:est_k_decay}
For $0 \leq a < 1/2$,  we have
\[
 \int_{\R^3} |k(\z, \z_*)| e^{- a |\z_*|^2}\,d\z_* \leq \frac{C_a}{1 + |\z|} e^{- a |\z|^2}
\]
for all $\z \in \R^3$.
\end{lemma}

Since $1 \leq (|\z_*|^{-1} + |\z_*|)/2$ for $\z_* \in \R^3 \setminus \{ 0 \}$,  Lemma \ref{lem:est_k_0}  implies the boundedness of the integral operator $K: L^\infty_\alpha(\Omega \times \R^3) \to L^\infty_\alpha(\Omega \times \R^3)$. Namely, we have the following  corollary.
\begin{corollary} \label{cor:bounded_K}
Assume $0 \leq \alpha < 1/2$. For $h \in L^\infty_\alpha(\Omega \times \R^3)$, we have
\[
\| K h \|_{L^\infty_\alpha(\Omega \times \R^3)} \leq C_\a \| h \|_{L^\infty_\alpha(\Omega \times \R^3)}.
\]
\end{corollary}

We also give estimates for the derivative of $k$. Inferring from \eqref{AD} and \eqref{factorE}, we see that
\[
\left| \nabla_\z k(\z, \z_*) \right| \leq C_\delta \frac{1 + |\z|}{|\z - \z_*|^2} e^{-a |\z|^2} e^{-\a_{1, a, \delta} |\z - \z_*|^2} e^{-(1-\delta) \left( \frac{(\z - \z_*) \cdot \z_*}{|\z - \z_*|} + \alpha_{2, a, \delta} |\z - \z_*| \right)^2} e^{a |\z_*|^2}.
\]
Then, we can show the following estimates.

\begin{lemma} \label{lem:est_Kdv_alpha}
For $0 \leq \a < 1/2$ and $0 \leq \sigma < 1$, we have
\begin{align}
& \int_{\R^3} |\nabla_\z k(\z, \z_*)| e^{-a |\z_*|^2}\,d\z_* \leq C_a e^{-a |\z|^2}, \label{Gk11}\\
& \int_{\R^3} |\nabla_{\z{_*}} k(\z, \z_*)| e^{-a |\z_*|^2}\,d\z_* \leq C_a e^{-a |\z|^2}, \label{Gk13}\\
& \int_{\R^3} \frac{1 + |\z_*|}{|\z_*|^\sigma} | \nabla_\z  k(\z, \z_*)| e^{- a |\z_*|^2}\,d\z_* \leq C_a (1 + |\z|) e^{- a |\z|^2} \label{Gk12}
\end{align}
for all $\z \in \R^3$.
\end{lemma}

\begin{proof}
In the same way as in the proof of Lemma \ref{lem:est_k_0}, by introducing the spherical coordinates, we have
\begin{align*}
&\int_{\R^3} |\nabla_\z k(\z, \z_*)| e^{-a |\z_*|^2}\,d\z_*\\
\leq& C (1 + |\z|) e^{-a |\z|^2} \int_{\R^3} \frac{1}{|\z - \z_*|^2} e^{-\a_{1, a, \delta} |\z - \z_*|^2} e^{-(1-\delta) \left( \frac{(\z - \z_*) \cdot \z_*}{|\z - \z_*|} + \alpha_{2, a, \delta} |\z - \z_*| \right)^2}\,d\xi_*\\
=& C (1 + |\z|) e^{-a |\z|^2} \int_0^\infty \int_{-1}^1 e^{-\a_{1, a, \delta} r^2} e^{-(1 - \delta) (|\z| t - \a_{2, a, \delta} r)^2}\,dtdr\\
\leq& C e^{-a |\z|^2}.
\end{align*}
Thus, the estimate \eqref{Gk11} is proved.

Due to the estimate \eqref{AE'}, we have
\begin{align*}
&\left| \nabla_{\z_*} k(\z, \z_*) \right|\\ 
\leq& C_\delta \frac{1 + |\z_*|}{|\z - \z_*|^2} e^{-a |\z|^2} e^{-\a_{1, a, \delta} |\z - \z_*|^2} e^{-(1-\delta) \left( \frac{(\z - \z_*) \cdot \z_*}{|\z - \z_*|} + \alpha_{2, a, \delta} |\z - \z_*| \right)^2} e^{a |\z_*|^2}\\
\leq& C_\delta \frac{1 + |\z|}{|\z - \z_*|^2} e^{-a |\z|^2} e^{-\a_{1, a, \delta} |\z - \z_*|^2} e^{-(1-\delta) \left( \frac{(\z - \z_*) \cdot \z_*}{|\z - \z_*|} + \alpha_{2, a, \delta} |\z - \z_*| \right)^2} e^{a |\z_*|^2}\\
&+ C_\delta \frac{1}{|\z - \z_*|} e^{-a |\z|^2} e^{-\a_{1, a, \delta} |\z - \z_*|^2} e^{-(1-\delta) \left( \frac{(\z - \z_*) \cdot \z_*}{|\z - \z_*|} + \alpha_{2, a, \delta} |\z - \z_*| \right)^2} e^{a |\z_*|^2}.
\end{align*}
Thus, the estimate \eqref{Gk13} is obtained by the combination of the estimate \eqref{Gk11} and the estimate for $k$.

For the estimate \eqref{Gk12}, it suffices to show that
\[
\int_{\R^3} \frac{1}{|\z_*|^\sigma} | \nabla_\z  k(\z, \z_*)| e^{- a |\z_*|^2}\,d\z_* \leq C_a (1 + |\z|) e^{- a |\z|^2}
\]
and
\[
\int_{\R^3} |\z_*|^{1- \sigma} | \nabla_\z  k(\z, \z_*)| e^{- a |\z_*|^2}\,d\z_* \leq C_a (1 + |\z|) e^{- a |\z|^2}.
\]
The first inequality is obtained by
\begin{align*}
&\int_{\R^3} \frac{1}{|\z_*|^\sigma} |\nabla_\z k(\z, \z_*)| e^{- a |\z_*|^2}\,d\z_*\\
\leq& C_a (1 + |\z|) e^{- a |\z|^2} \int_{\R^3} \frac{1}{|\z_*|^\sigma |\z - \z_*|^2} e^{- \a_{1, a, \delta} |\z - \z_*|^2}\,d\z_*\\
\leq& C_a (1 + |\z|) e^{- a |\z|^2} \left( \int_{\{ |\z_*| \leq |\z - \z_*| \}} \frac{1}{|\z_*|^{2 + \sigma}} e^{- \a_{1, a, \delta} |\z_*|^2}\,d\z_* \right.\\
&\left.  + \int_{\{ |\z_*| > |\z - \z_*| \}} \frac{1}{|\z - \z_*|^{2+ \sigma}} e^{- \a_{1, a, \delta} |\z - \z_*|^2}\,d\z_* \right)\\
\leq& C_a (1 + |\z|) e^{- a |\z|^2}.
\end{align*}
For the second inequality, we notice that 
\[
|\z_*|^{1 - \sigma} \leq 1 + |\z_*| \leq 1 + |\z| + |\z - \z_*|.
\]
Thus, we have
\begin{align*}
&\int_{\R^3} |\z_*|^{1- \sigma} \left| \nabla_\z k(\z, \z_*) \right| e^{-a |\z_*|^2}\,d\z_*\\ 
\leq& C_\delta (1 + |\z|)^2 e^{-a|\z|^2} \int_{\R^3} \frac{1}{|\z - \z_*|^2} e^{-\a_{1, a, \delta} |\z - \z_*|^2} e^{-(1-\delta) \left( \frac{(\z - \z_*) \cdot \z_*}{|\z - \z_*|} + \alpha_{2, a, \delta} |\z - \z_*| \right)^2}\,d\z_*\\
&+ C_\delta (1+ |\z|) e^{-a |\z|^2} \int_{\R^3} \frac{1}{|\z - \z_*|} e^{-\a_{1, a, \delta} |\z - \z_*|^2} e^{-(1-\delta) \left( \frac{(\z - \z_*) \cdot \z_*}{|\z - \z_*|} + \alpha_{2, a, \delta} |\z - \z_*| \right)^2}\,d\z_*\\
\leq& C_\delta (1+ |\z|) e^{-a |\z|^2}.
\end{align*}
This completes the proof.
\end{proof}

\begin{corollary} \label{cor:est_Kdv_alpha_op}
For $0 \leq \alpha < 1/2$, we have
\[
\| \nabla_\z K h \|_{L^\infty_\alpha(\Omega \times \R^3)} \leq C_\a \| h \|_{L^\infty_\alpha(\Omega \times \R^3)}
\]
for all $h \in L^\infty_\alpha(\Omega \times \R^3)$.
\end{corollary}

We introduce some estimates in order to discuss regularity of the $III$ part.

\begin{lemma} \label{lem:K_Lip}
Let $0 \leq \alpha < 1/2$. Then, for any $f \in L^\infty_\a(\O \times \R^3)$, we have
\[
|Kf(x, \rho \hat{\z_1}) - Kf(x, \rho \hat{\z_2})| \leq C \| f \|_{L^\infty_\a(\O \times \R^3)} \rho |\hat{\z_1} - \hat{\z_2}| e^{- \a \rho^2}
\]
for a.e. $x \in \O$, $\rho > 0$ and $\hat{\z_1}, \hat{\z_2} \in S^2$.
\end{lemma}

\begin{proof}
Without loss of generality, by introducing the spherical coordinates, we may assume $\hat{\z_1} = (1, 0, 0)$ and $\hat{\z_2} = (\cos \theta, \sin \theta, 0)$ with $0 \leq \theta \leq \pi$. The estimate holds when $\theta = 0$. Thus, in what follows, we discuss the case $\theta > 0$.

Set $\hat{\z}(t) = (\cos t, \sin t, 0)$ for  $0 \leq t \leq \theta \}$. Then, we have
\begin{align*}
Kf(x, \rho \hat{\z_2}) - Kf(x, \rho \hat{\z_1}) =& \int_0^\theta \frac{d}{dt} Kf(x, \rho \hat{\z}(t))\,dt\\
=& \int_0^\theta \rho \hat{\z}(t)' \cdot \nabla_\z Kf(x, \rho \hat{\z}(t))\,dt.
\end{align*}
By the definition of $\hat{\z}(t)$, we see that $|\hat{\z}(t)'| = 1$ for all $0 \leq t \leq \theta$. Also, by Corollary \ref{cor:est_Kdv_alpha_op}, we obtain
\[
|\nabla_\z Kf(x, \rho \hat{\z}(t))| \leq \| \nabla_\z Kf \|_{L^\infty_\a(\O \times \R^3)} e^{-\a |\rho \hat{\z}(t)|^2} \leq C_\a \| f \|_{L^\infty_\a(\O \times \R^3)} e^{-\a \rho^2}
\]
for all $0 \leq t \leq \theta$. Thus, we have
\[
|Kf(x, \rho \hat{\z_1}) - Kf(x, \rho \hat{\z_2})| \leq C_\a \| f \|_{L^\infty_\a(\O \times \R^3)} \rho \theta e^{-\a \rho^2}.
\]

The conclusion follows from the following fact
\[
\theta \leq \pi \sin \frac{\theta}{2} = \frac{\pi}{2} |\hat{\z_1} - \hat{\z_2}|.
\]
This completes the proof.
\end{proof}

With a slight modification, we obtain the following estimate.

\begin{lemma} \label{lem:k_Lip}
For $\hat{\z_0}, \hat{\z_1} \in S^2$, let $\theta$ be the  arc length of the shortest arc on $S^2$ that connects $\hat{\z_0}$ and $\hat{\z_1}$. We parametrize this arc by $\tilde{\z}(u)$ for $0 \leq u \leq 1$  so that $\tilde{\z}(0) = \hat{\z_0}$, $\tilde{\z}(1) = \hat{\z_1}$ and $|\tilde{\z}'(u)| = \theta$. Then, for any $ \delta \in(0, 1)$ and $\rho > 0$, we have
\[
|k(\z, \rho \hat{\z_1}) - k(\z, \rho \hat{\z_0})| \leq C_\delta \rho(1 + \rho) |\hat{\z_1} - \hat{\z_0}| \int_0^1 \frac{1}{|\z - \rho \tilde{\z}(u)|^2} E_\delta(\rho \tilde{\z}(u), \z)\,du
\]
for all  $\z \in \R^3 \setminus \{\rho \tilde{\z}(u) \mid \,0 \leq u \leq 1\}$.
\end{lemma}


\begin{proof}
By the fundamental theorem of calculus, we have
\begin{align*}
|k(\z, \rho \tilde{\z_1}) - k(\z, \rho \tilde{\z_0})| =& \left| \int_0^1 \frac{d}{du} k(\z, \rho \tilde{\z}(u))\,du \right|\\
\leq& \int_0^1 |(\nabla_{\z_*} k)(\z, \rho \tilde{\z}(u)) \rho \tilde{\z}'(u)|\,du\\
\leq& C_\delta \rho \theta \int_0^1 \frac{1 + \rho}{|\z - \rho \tilde{\z}(u)|^2} E_\delta(\z, \rho \tilde{\z}(u))\,du.
\end{align*}
The conclusion follows from the estimate $\theta \leq C |\hat{\z_1} - \hat{\z_0}|$.
\end{proof}

\subsection{Miscellaneous Estimates}
As has been proved and used in \cite{CHKS2}, the following four lemmas will be used frequently in our nonlinear estimates as the cross section $B$ satisfies the assumption \eqref{assumption_B1}. 
\begin{lemma} \label{lem:estimate_nonlin_1}
Let $\beta > 0$ and $0 \leq \gamma \leq 1$. Then, we have
\[
\int_{\R^3} \left(1+\frac{1}{|\z_*|}\right) e^{-\beta |\z_*|^2} |\z - \z_*|^\gamma\,d\z_* \leq C (1 + |\z|)^\gamma
\]
for all $\z \in \R^3$.
\end{lemma}

\begin{lemma} \label{lem:estimate_nonlin_2}
Let $\beta > 0$ and $0 \leq \gamma \leq 1$. Then, we have
\[
\int_{\R^3} e^{-\beta |\z_*|^2} |\z - \z_*|^{\gamma-1}\,d\z_* \leq C
\]
for all $\z \in \R^3$.
\end{lemma}

\begin{lemma} \label{lem:estimate_nonlin_3}
Let $\Omega$ be an open bounded domain in $\R^3$ and $0 \leq \gamma \leq 1$. Then, we have
\[
(1 + |\z|)^\gamma \min \left\{ 1, \frac{\diam(\Omega)}{|\z|} \right\} \leq 1 + \diam(\Omega)
\]
for all $\z \in \R^3 \setminus \{0\}$.
\end{lemma}

Recall that the collision operator $Q$  is defined as
\[
Q(F, G) := \int_{\R^3} \int_0^{2\pi} \int_0^{\frac{\pi}{2}} [F(\z') G(\z_*') - F(\z) G(\z_*)] B(|\z - \z_*|, \theta)\,d\theta d\phi d\z_*,
\]
where
\begin{equation}
\label{theta2222}
\left\{
\begin{aligned}
\z' :=& \z + ((\z_* - \z) \cdot \omega) \omega, \quad \z_*' := \z_* - ((\z_* - \z) \cdot \omega) \omega,\\
\omega :=& \cos \theta \frac{\z_* - \z}{|\z_* - \z|} + (\sin \theta \cos \phi) e_2 + (\sin \theta \sin \phi) e_3.
\end{aligned}
\right.
\end{equation}
Here, $0 \leq \phi \leq 2\pi$, $0 \leq \theta \leq \pi/2$, and $e_2$ and $e_3$ are unit vectors such that the pair $\{ (\z_* - \z)/|\z_* - \z|, e_2, e_3 \}$ forms an orthonormal basis in $\R^3$. Throughout this article, we consider the following cross section $B$: 
\[
B(|\z - \z_*|,\theta) = B_0 |\z - \z_*|^{\gamma}\sin{\theta}\cos{\theta}, \,\, \qquad  0 \leq \gamma \leq 1,
\]
for some $B_0 > 0$. In Section 7,  to estimate  the nonlinear terms  $|\nabla_x \Gamma(h_1, h_2)(x, \z)|$ and 
$|\nabla_\z \Gamma(h_1, h_2)(x, \z)|$, we encounter the following integrals
\[
\int_0^{2\pi} \int_0^{\frac{\pi}{2}} \frac{1}{|\z'|}  \sin{\theta}\cos{\theta} d\theta d\phi, \quad \int_0^{2\pi} \int_0^{\frac{\pi}{2}} \frac{1}{|\z_*'|}  \sin{\theta}\cos{\theta} d\theta d\phi.
\]

\begin{lemma} \label{lem:int_sigma}
In the setting of \eqref{collision1}-\eqref{assumption_B1}, for 
 $\z, \, \z_* \in \R^3$ with $\z \neq \z_*$, we have
\begin{equation}\label{Estimate 1/z' S_2}
 \int_0^{2\pi} \int_0^{\frac{\pi}{2}} \frac{1}{|\z'|}  \sin{\theta}\cos{\theta} d\theta d\phi= 2\pi \min \left\{ \frac{1}{|\z+\z_*|},\frac{1}{|\z-\z_*|} \right\}
\end{equation}
and
\begin{equation}\label{Estimate 1/z_*' S_2}
\int_0^{2\pi} \int_0^{\frac{\pi}{2}} \frac{1}{|\z_*'|}  \sin{\theta}\cos{\theta} d\theta d\phi = 2\pi \min \left\{ \frac{1}{|\z + \z_*|},\frac{1}{|\z - \z_*|} \right\}.
\end{equation}
\end{lemma}

\begin{proof}
We introduce the sigma formulation for $\z'$ and $\z_*'$:
\[
\left\{
\begin{aligned}
& \z'(\sigma) := \frac{\z + \z_*}{2} + \frac{|\z_* - \z|}{2} \sigma, \, \z_*'(\sigma) := \frac{\z + \z_*}{2} - \frac{|\z_* - \z|}{2} \sigma, \\
 & \sigma := \cos{\psi} \frac{\z_* - \z}{|\z_* - \z|}+ (\sin{\psi}\cos{\phi}) e_2 + (\sin{\psi}\sin{\phi}) e_3.
 \end{aligned}
 \right.
 \]
Here, $0 \leq \phi < 2\pi$, $0 \leq \psi \leq \pi$, and $e_2$ and $e_3$ are unit vectors such that the pair $\{ (\z_* - \z)/|\z_* - \z|, e_2, e_3 \}$ forms an orthonormal basis in $\R^3$. Compared with formulation \eqref{theta2222}, we note that $\psi = 2\theta$. 
Hence,
\begin{equation*}
\begin{aligned}
 \int_0^{2\pi} \int_0^{\frac{\pi}{2}} \frac{1}{|\z'|}  \sin{\theta}\cos{\theta} d\theta d\phi=& \frac{1}{2} \int_0^{2\pi} \int_0^{\frac{\pi}{2}} \frac{1}{|\z'|}  \sin{2\theta} d\theta d\phi\\
 = &\frac{1}{4} \int_0^{2\pi} \int_0^{\pi} \frac{1}{|\z'|}  \sin{\psi} d\psi d\phi \\
 =& \int_{\mathbb{S}^{2}}\frac{1}{| \z'(\sigma) |}\,d\Sigma(\sigma). 
 \end{aligned}
\end{equation*}
With this understanding, to show \eqref{Estimate 1/z' S_2} and \eqref{Estimate 1/z_*' S_2}, it is equivalent to prove
\begin{equation}
\label{sigmaEstimate1/z'S_2}
\int_{\mathbb{S}^{2}}\frac{1}{| \z'(\sigma) |}\,d\Sigma(\sigma) = 8\pi \min \left\{ \frac{1}{|\z+\z_*|},\frac{1}{|\z-\z_*|} \right\}
\end{equation}
and
\begin{equation}
\label{sigmaEstimate1/z_*'S_2}
\int_{\mathbb{S}^{2}}\frac{1}{| \z_*'(\sigma) |}\,d\Sigma(\sigma) = 8\pi \min \left\{ \frac{1}{|\z + \z_*|},\frac{1}{|\z - \z_*|} \right\}.
\end{equation}
Direct calculation shows
\begin{align*}
&\int_{\mathbb{S}^{2}}\frac{1}{| \z'(\sigma) |}\,d\Sigma(\sigma)\\
=& \int_{0}^{2\pi}\int_{0}^{\pi}\frac{1}{\sqrt{|\frac{\z + \z_*}{2} |^{2}+|\frac{\z-\z_*}{2}|^{2}+2 |\frac{\z + \z_*}{2}| |\frac{\z-\z_*}{2}| \cos{\psi}}}\sin{\psi}\,d\psi d\phi\\
=& \int_{0}^{2\pi}\int_{-1}^{1}\frac{1}{\sqrt{|\frac{\z + \z_*}{2} |^{2}+|\frac{\z-\z_*}{2}|^{2}+2|\frac{\z + \z_*}{2}|| \frac{\z-\z_*}{2}|z}}\,dz d\phi\\
=& 4\pi \int_{-1}^{1}\frac{1}{\sqrt{|\z + \z_*|^{2}+|\z-\z_*|^{2}+2|\z + \z_*| |\z-\z_*| z}}\,dz\\
=& 4\pi \frac{|\z + \z_*| + |\z-\z_*| - \left| |\z + \z_*| - |\z-\z_*| \right|}{|\z + \z_*| |\z-\z_*|}\\
=& 8\pi \min \left\{ \frac{1}{|\z + \z_*|},\frac{1}{|\z - \z_*|} \right\}.
\end{align*}
Thus, the identity \eqref{sigmaEstimate1/z'S_2} holds. The identity \eqref{sigmaEstimate1/z_*'S_2} can be shown in the same way. This completes the proof of the lemma.
\end{proof}
\noindent Lemma~\ref{lem:int_sigma} has also been used in the proof of Lemma 4.3 in \cite{CHKS2}. 

\section{Regularity issues : the smoothing effects of the integral operators} \label{sec:Hol}

In this section, we investigate the smoothing effects of the following operators:
\begin{align*}
I(x, \z) :=& Jf_0(x, \z) + S_\Omega \phi(x, \z),\\
II(x, \z) :=& S_\Omega K(Jf_0 + S_\Omega \phi)(x, \z) = S_\Omega K I(x, \z),\\
III(x, \z) :=& (S_\Omega K)^2 f(x, \z),
\end{align*}
where $J$, $S_\O$ and $K$ are operators defined in \eqref{J}, \eqref{S} and \eqref{K}, respectively.

\subsection{H\"older regularity of $I$} \label{subsec:Hol_I}

The following proposition is a direct result of Corollary \ref{cor:bound_S_Calpha} and Proposition \ref{prop:bound_J_Calpha}.

\begin{proposition} \label{prop:bounded_I}
Let $f_0 \in L^\infty_\a(\Gamma^-)$ and $\phi \in L^\infty_{\a, 1}(\O \times \R^3)$. Then, we have
\[
\| I \|_{L^\infty_\a(\O \times \R^3)} \leq C \left( \| f_0 \|_{L^\infty_\a(\Gamma^-)} + \| \phi \|_{L^\infty_{\a, 1}(\O \times \R^3)} \right).
\]
\end{proposition}
The following lemma is a direct result of Lemma \ref{lem:J_Hol} and Lemma \ref{lem:S_Hol}.
\begin{lemma} \label{lem:Hol_I}
Assume $1/2 < \sigma \leq 1$ and $\a \geq 0$. Let $f_0 \in \cB_\a$ and $\phi \in \cC^\sigma_\a$. Then, we have
\begin{align*}
&|I(x, \z) - I(y, \z)|\\ 
\leq& C_\sigma \left( \| f_0 \|_{\cB_\a} + \| \phi \|_{\cC^\sigma_\a} \right) \left( N(x, y, \z)^{-1} + d_{x, y}^{\sigma - 1} \right) \frac{1 + |\z|}{|\z|^\sigma} |x - y|^\sigma e^{-\a |\z|^2}
\end{align*}
for all $x, y \in \O$ and $\z \in \R^3 \setminus \{ 0 \}$.
\end{lemma}

\subsection{H\"older regularity of $II$} \label{subsec:Hol_II}

Combining Corollary \ref{cor:bound_S_Calpha}, Corollary \ref{cor:bounded_K}, and Proposition \ref{prop:bounded_I}, we obtain the following estimate.

\begin{proposition} \label{prop:bounded_II}
Let $f_0 \in L^\infty_\a(\Gamma^-)$ and $\phi \in L^\infty_{\a, 1}(\O \times \R^3)$. Then, we have
\[
\| II \|_{L^\infty_\a(\O \times \R^3)} \leq C \left( \| f_0 \|_{L^\infty_\a(\Gamma^-)} + \| \phi \|_{L^\infty_{\a, 1}(\O \times \R^3)} \right).
\]
\end{proposition}

\begin{lemma} \label{lem:Hol_KI}
Assume $1/2 < \sigma \leq 1$ and $0 \leq \a < 1/2$. Let $f_0 \in \cB_\a$ and $\phi \in \cC^\sigma_\a$. Then, we have
\begin{align*}
&\int_{\R^3} |k(\z, \z_*)| |I(x, \z_*) - I(y, \z_*)|\,d\z_*\\ 
\leq& C_\sigma \left( \| f_0 \|_{\cB_\a} + \| \phi \|_{\cC^\sigma_\a} \right) d_{x, y}^{-\frac{1}{2}} |x - y|^\sigma e^{-\a |\z|^2}
\end{align*}
for all $x, y \in \O$ and $\z \in \R^3$.
\end{lemma}



\begin{proof}
Applying Lemma \ref{lem:Hol_I}, Proposition \ref{prop:N1/2} and Lemma \ref{lem:est_k_0}, we have
\begin{align*}
&\int_{\R^3} |k(\z, \z_*)| |I(x, \z_*) - I(y, \z_*)|\,d\z_*\\ 
\leq& C_\sigma \left( \| f_0 \|_{\cB_\a} + \| \phi \|_{\cC^\sigma_\a} \right) |x - y|^\sigma\\ 
&\times \int_{\R^3} |k(\z, \z_*)| \left( N^-(x, y, \z_*) + d_{x, y}^{\sigma - 1} \right) \frac{1 + |\z_*|}{|\z_*|^\sigma} e^{-\a |\z_*|^2}\,d\z_*\\
\leq& C_\sigma \left( \| f_0 \|_{\cB_\a} + \| \phi \|_{\cC^\sigma_\a} \right) (d_{x, y}^{-\frac{1}{2}} + d_{x, y}^{\sigma - 1} )|x - y|^\sigma \int_{\R^3} \frac{1 + |\z_*|}{|\z_*|^\sigma} |k(\z, \z_*)| e^{-\a |\z_*|^2}\,d\z_*\\
\leq& C_\sigma \left( \| f_0 \|_{\cB_\a} + \| \phi \|_{\cC^\sigma_\a} \right) d_{x, y}^{-\frac{1}{2}} |x - y|^\sigma e^{-\a |\z|^2}.
\end{align*}
This completes the proof.
\end{proof}

Lemma~\ref{lem:Hol_KI}  implies that 
\begin{equation} \label{est:Hol_KI}
|KI(x, \z) - KI (y, \z)| \leq C \left( \| f_0 \|_{\cB_\a} + \| \phi \|_{\cC^\sigma_\a} \right) d_{x, y}^{-\frac{1}{2}} |x - y|^\sigma e^{-\a |\z|^2}
\end{equation}
for all $x, y \in \O$ and $\z \in \R^3$.

\begin{lemma} \label{lem:Hol_II}
For $1/2 < \sigma \leq 1$ and $0 \leq \a < 1/2$, let $f_0 \in \cB_\a$ and $\phi \in \cC^\sigma_\a$. Then, we have
\begin{align*}
&|II(x, \z) - II(y, \z)|\\ 
\leq& C_\sigma \left( \| f_0 \|_{\cB_\a} + \| \phi \|_{\cC^\sigma_\a} \right) \left( N(x, y, \z)^{-1} + d_{x, y}^{\sigma - 1} \right) \frac{1 + |\z|}{|\z|^\sigma}  |x - y|^\sigma e^{-\a |\z|^2}
\end{align*}
for all $x, y \in \O$ and $\z \in \R^3 \setminus \{ 0 \}$.
\end{lemma}

\begin{proof}
Without loss of generality, we may assume $\tau_-(x, \z) \leq \tau_-(y, \z)$. Hence,
\begin{align*}
|II(x, \z) - II(y, \z)| =& |S_\O K I (x, \z) - S_\O K I (y, \z)|\\
\leq& \int_0^{\tau_-(x, \z)} e^{- \nu_0 t} |KI(x - t\z, \z) - KI(y - t\z, \z)|\,dt\\ 
&+ \int_{\tau_-(x, \z)}^{\tau_-(y, \z)} e^{- \nu_0 t} |KI(y - t\z, \z)|\,dt.
\end{align*}
For the first term on the right hand side, the estimate \eqref{est:Hol_KI} with Lemma~\ref{lem:S1/2} yields
\begin{align*}
&\int_0^{\tau_-(x, \z)} e^{- \nu_0 t} |KI(x - t\z, \z) - KI(y - t\z, \z)|\,dt\\
\leq& C \left( \| f_0 \|_{\cB_\a} + \| \phi \|_{\cC^\sigma_\a} \right) |x - y|^\sigma e^{-\a |\z|^2} \int_0^{\tau_-(x, \z)} e^{-\nu_0 t} d_{x - t\z, y - t\z}^{-\frac{1}{2}} \,dt\\
\leq& \frac{C_\sigma}{|\z|^\sigma} \left( \| f_0 \|_{\cB_\a} + \| \phi \|_{\cC^\sigma_\a} \right) d_{x, y}^{\sigma - 1} |x - y|^\sigma e^{- \a |\z|^2}.
\end{align*}
For the second term, by Proposition \ref{prop:bounded_I} and Lemma \ref{lem:exp_inter}, we get 
\begin{align*}
&\int_{\tau_-(x, \z)}^{\tau_-(y, \z)} e^{- \nu_0 t} |KI(y - t\z, \z)|\,dt\\ 
\leq& C \| I \|_{L^\infty_\a(\O \times \R^3)} e^{-\a |\z|^2} \int_{\tau_-(x, \z)}^{\tau_-(y, \z)} e^{- \nu_0 t}\,dt\\
\leq& C_\sigma \left( \| f_0 \|_{\cB_\a} + \| \phi \|_{\cC^\sigma_\a} \right) N(x, y, \z)^{-1} \frac{1 + |\z|}{|\z|^\sigma} |x - y|^\sigma e^{-\a |\z|^2}. 
\end{align*}
Noting that  $1+ |\z|^{-\sigma} \le  2(1+ |\z|)/|\z|^{\sigma}$ for $\sigma \in [0,1]$ and combining the above two estimates, we obtain the lemma. This completes the proof.
\end{proof}

\subsection{H\"older regularity of $III$} \label{subsec:Hol_III}

We rewrite the formula of the $III$ part as $III = S_\O G$, where
\begin{align*}
G(x, \z) :=& K S_\O K f(x, \z)\\
=& \int_{\R^3} k(\z, \z_*) \int_0^{\tau_-(x, \z_*)} e^{-\nu(|\z_*|) t} Kf (x - t\z_*, \z_*)\,dtd\z_*.
\end{align*}
By a specific change of  variables, see  \cite{CHK}, we can express $G(x, \z)$ as
\[
G(x, \z) = \int_0^\infty \int_\O k \left( \z, \rho \frac{x - z}{|x - z|} \right) e^{-\nu(\rho) \frac{|x - z|}{\rho}} Kf \left( z, \rho \frac{x - z}{|x - z|} \right) \frac{\rho}{|x - z|^2}\,dzd\rho.
\]
By the Lebesgue dominated convergence theorem, we regard $G$ as a continuous function on $\overline{\O} \times \R^3$.

\begin{proposition} \label{prop:Hol_G}
For $0 \leq \a < 1/2$, let $f \in L^\infty_\a(\O \times \R^3)$. Then, we have
\[
|G(x, \z) - G(y, \z)| \leq C \| f \|_{L^\infty_\a(\O \times \R^3)} |x - y| (1 + |\log |x - y||) e^{-\a |\z|^2}
\]
for all $x, y \in \overline{\O}$ and $\z \in \R^3$.
\end{proposition}
\begin{remark}
Proposition~\ref{prop:Hol_G} is a modification of  Proposition 7.1 in \cite{CHK}, see also  Lemma 2.1 in \cite{RegularChen}. For readers' convenience, we provide a detailed proof.
\end{remark}

\begin{proof}
Applying the triangle inequality,  we obtain
\[
|G(x, \z) - G(y, \z)| \leq \sum_{i = 1}^4 G_i,
\]
where 
\begin{align*}
G_1 :=& \int_0^\infty \int_\O \left| k \left( \z, \rho \frac{x - z}{|x - z|} \right) \right| \frac{\rho}{|x - z|^2} \\
&\times \left| Kf \left( z, \rho \frac{x - z}{|x - z|} \right) - Kf \left( z, \rho \frac{y - z}{|y - z|} \right) \right| \,dz d\rho,\\
G_2 :=& \int_0^\infty \int_\O \left| Kf \left( z, \rho \frac{y - z}{|y - z|} \right) \right| \frac{\rho}{|x - z|^2} \\
&\times \left| k \left( \z, \rho \frac{x - z}{|x - z|} \right) - k \left( \z, \rho \frac{y - z}{|y - z|} \right) \right| \,dz d\rho,\\
G_3 :=& \int_0^\infty \int_\O \left| Kf \left( z, \rho \frac{y - z}{|y - z|} \right) \right| \left| k \left( \z, \rho \frac{y - z}{|y - z|} \right) \right|\\
&\times \rho \left| e^{-\nu(\rho) \frac{|x - z|}{\rho}} - e^{-\nu(\rho) \frac{|y - z|}{\rho}} \right| \frac{1}{|x - z|^2} \,dz d\rho,\\
G_4 :=& \int_0^\infty \int_\O \left| Kf \left( z, \rho \frac{y - z}{|y - z|} \right) \right| \left| k \left( \z, \rho \frac{y - z}{|y - z|} \right) \right| \rho\\
&\times \left| \frac{1}{|x - z|^2} - \frac{1}{|y - z|^2}  \right| \,dz d\rho.
\end{align*}

We first give an estimate for $G_1$. We split the domain of integration into two parts: $\O_1 := \O \cap B(x, 2|x - y|)$ and $\O_2 := \O \setminus B(x, 2|x - y|)$. And we denote the corresponding integrals by $G_{1, 1}$ and $G_{1, 2}$ respectively. 

For the $G_{1, 1}$, we apply Corollary \ref{cor:bounded_K} to obtain
\begin{align*}
&\left| Kf \left( z, \rho \frac{x - z}{|x - z|} \right) - Kf \left( z, \rho \frac{y - z}{|y - z|} \right) \right|\\ 
\leq& \left| Kf \left( z, \rho \frac{x - z}{|x - z|} \right) \right| + \left| Kf \left( z, \rho \frac{y - z}{|y - z|} \right) \right|\\
\leq& C \| K f \|_{L^\infty_\a(\O \times \R^3)} e^{- \a \rho^2}\\
\leq& C \| f \|_{L^\infty_\a(\O \times \R^3)} e^{- \a \rho^2}.
\end{align*}
Hence, we have
\[
G_{1, 1} \leq C \| f \|_{L^\infty_\a(\O \times \R^3)} \int_0^\infty \int_{\O_1} \left| k \left( \z, \rho \frac{x - z}{|x - z|} \right) \right| \frac{\rho}{|x - z|^2} e^{- \a \rho^2}\,dz d\rho.
\]
Introducing the spherical coordinates $z = x - r \hat{\z_*}$ with $0 < r < 2|x - y|$ and $\hat{\z_*} \in S^2$, we have
\begin{align*}
&\int_0^\infty \int_{\O_1} \left| k \left( \z, \rho \frac{x - z}{|x - z|} \right) \right| \frac{\rho}{|x - z|^2} e^{- \a \rho^2}\,dz d\rho\\
 \leq& \int_0^\infty \int_0^{2|x - y|} \int_0^\pi \int_0^{2\pi} \left| k ( \z, \rho \hat{\z_*} ) \right| e^{- \a \rho^2} \rho \sin \theta \,d\phi d\theta dr d\rho\\
 =& 2|x - y| \int_{\R^3} \frac{1}{|\z_*|} \left| k (\z, \z_*)\right| e^{-\a |\z_*|^2} \,d\z_*\\
\leq& C |x - y| e^{- \a |\z|^2}. 
\end{align*}
Therefore, we have
\[
G_{1, 1} \leq C \| f \|_{L^\infty_\a(\O \times \R^3)} |x - y| e^{-\a |\z|^2}.
\]
For the $G_{1, 2}$, applying Lemma \ref{lem:K_Lip}, we obtain
\begin{align*}
&\left| Kf \left( z, \rho \frac{x - z}{|x - z|} \right) - Kf \left( z, \rho \frac{y - z}{|y - z|} \right) \right|\\
\leq& C \| f \|_{L^\infty_\a(\O \times \R^3)} \rho \left| \frac{x - z}{|x - z|} - \frac{y - z}{|y - z|} \right| e^{- \a \rho^2}\\
\leq& C \| f \|_{L^\infty_\a(\O \times \R^3)} \rho \frac{|x - y|}{|x - z|} e^{- \a \rho^2},
\end{align*}
where we have used the fact : 
\begin{align*}
\frac{x - z}{|x - z|} - \frac{y - z}{|y - z|} = \frac{x - y}{|x - z|} + \frac{|y-z| -|x-z| }{|x - z||y-z|} (y-z)
\end{align*}
and 
\begin{align*}
\Big| |y-z| -|x-z| \Big| \le |x-y|
\end{align*}
for $z \in \O_2$.
Thus, we have
\[
G_{1, 2} \leq C  \| f \|_{L^\infty_\a(\O \times \R^3)} |x - y| \int_0^\infty \int_{\O_2} \left| k \left( \z, \rho \frac{x - z}{|x - z|} \right) \right| \frac{\rho^2}{|x - z|^3} e^{- \a \rho^2}  \,dz d\rho.
\]
Using the same spherical coordinates as we have used in the estimate of $G_{1, 1}$, we have
\begin{align*}
&\int_0^\infty \int_{\O_2} \left| k \left( \z, \rho \frac{x - z}{|x - z|} \right) \right| \frac{\rho^2}{|x - z|^3} e^{- \a \rho^2} \,dz d\rho\\
\leq& \left( \int_{2|x - y|}^{2 \diam(\O)} \frac{1}{r}\,dr \right) \left( \int_{\R^3} |k(\z, \z_*)| e^{- \a |\z_*|^2}\,d\z_* \right)\\
\leq& C \left( |\log |x - y|| + 1 \right) e^{- \a |\z|^2}.
\end{align*}
Thus, we obtain 
\[
G_{1, 2} \leq  C  \| f \|_{L^\infty_\a(\O \times \R^3)} |x - y| \left( |\log |x - y|| + 1 \right) e^{- \a |\z|^2}
\]
and hence
\[
G_1 \leq  C  \| f \|_{L^\infty_\a(\O \times \R^3)} |x - y| \left( |\log |x - y|| + 1 \right) e^{- \a |\z|^2}.
\]

We next estimate the other terms in the same manner. Namely, for $i = 2, 3, 4$ and $j = 1, 2$, we let $G_{i, j}$ denote the integral $G_i$ restricted on $\O_j$. Since $\O_1 \subset B(x, 2|x - y|) \subset B(y, 3|x - y|)$, the smallness of $\O_1$ yields
\[
G_{i, 1} \leq C \| f \|_{L^\infty_\a(\O \times \R^3)} |x - y| e^{- \a |\z|^2}
\]
for all $i = 2, 3, 4$. In what follows, we estimate $G_{i, 2}$ for $i = 2, 3, 4$.
Concerning $G_{2, 2}$, by Lemma \ref{lem:k_Lip}, we have
\begin{align*}
G_{2, 2} \leq C_\delta \| f \|_{L^\infty_\a(\O \times \R^3)} |x - y| \int_0^\infty \int_{\O_2} \frac{\rho^2(1 + \rho)}{|x - z|^3} \int_0^1 \frac{E_\delta(\z, \rho \tilde{\z}(u; z))}{|\z - \rho \tilde{\z}(u; z)|^2} e^{- \a \rho^2}\,du dz d\rho.
\end{align*}
We remark that, see Lemma~\ref{lem:k_Lip},  $\{\tilde{\z}(u; z) \mid 0 \leq u \leq 1\}$ is the shortest arc on $S^2$ that connects two unit vectors $(x - z)/|x - z|$ and $(y - z)/|y - z|$. 

Let $\tilde{x}(u)$ be the point on the line segment $\overline{xy}$ such that $\tilde{\z}(u) = (\tilde{x}(u) - z)/|\tilde{x}(u) - z|$. We change the order of integration and introduce the spherical coordinates for $z$ to obtain
\begin{align*}
&\int_0^\infty \int_{\O_2} \int_0^1 \frac{\rho^2(1 + \rho)}{|x - z|^3} \frac{E_\delta(\z, \rho \tilde{\z}(u; z))}{|\z - \rho \tilde{\z}(u; z)|^2} e^{- \a \rho^2}\,du dz d\rho\\
\leq& \int_0^1 \int_0^\infty \int_{|x - y|}^{\diam(\O)} \int_0^\pi \int_0^{2\pi} \frac{\rho^2(1 + \rho)}{r} \frac{1}{|\z - \rho \frac{\tilde{x}(u) - z(r, \theta, \phi)}{|\tilde{x}(u) - z(r, \theta, \phi)|}|^2}\\ 
&\times E_\delta \left( \z, \rho \frac{\tilde{x}(u) - z(r, \theta, \phi)}{|\tilde{x}(u) - z(r, \theta, \phi)|} \right) e^{- \a \rho^2} \sin \theta \,d\phi d\theta dr d\rho du\\
\leq& \int_0^1 \int_{|x - y|}^{\diam(\O)} \int_{\R^3} \frac{1 + |\z'|}{r} \frac{1}{|\z - \z'|^2} E_\delta ( \z, \z') e^{- \a |\z'|^2} \,d\z' dr du.
\end{align*}
Since $|\z'| \leq |\z - \z'| + |\z|$, we may follow proofs of Lemma \ref{lem:est_k_0} and Lemma \ref{lem:est_Kdv_alpha} to obtain
\[
\int_{\R^3} \frac{1 + |\z'|}{|\z - \z'|^2} E_\delta ( \z, \z') e^{- \a |\z'|^2} \,d\z' \leq C e^{- \a |\z|^2}.
\]
Thus, we have
\begin{align*}
&\int_0^1 \int_{|x - y|}^{\diam(\O)} \int_{\R^3} \frac{1 + |\z'|}{r} \frac{1}{|\z - \z'|^2} E_\delta ( \z, \z') e^{- \a |\z'|^2} \,d\z' dr du\\ 
\leq& C (|\log |x - y|| + 1) e^{- \a |\z|^2},
\end{align*}
or
\[
G_{2, 2} \leq C \| f \|_{L^\infty_\a(\O \times \R^3)} |x - y|  (|\log |x - y|| + 1) e^{- \a |\z|^2}. 
\]
We next estimate $G_{3, 2}$. We notice that
\[
|y - z| \leq |y - x| + |x - z| \leq \frac{3}{2} |x - z|
\]
and
\[
|y - z| \geq |x - z| - |x - y| \geq |x - y|
\]
for all $z \in \O_2$. Thus, we have
\begin{align*}
\left| e^{-\nu(\rho) \frac{|x - z|}{\rho}} - e^{-\nu(\rho) \frac{|y - z|}{\rho}} \right| =& \left| \frac{\nu(\rho)}{\rho} \int_{|y - z|}^{|x - z|} e^{- \frac{\nu(\rho)}{\rho} t}\,dt \right|\\
\leq& C \left| \int_{|y - z|}^{|x - z|} \frac{1}{t} \,dt \right|\\
\leq& C \max \left\{ \frac{1}{|x - z|}, \frac{1}{|y - z|} \right\} \left| |y - z| - |x - z| \right|\\
\leq& C \frac{|x - y|}{|y - z|}.
\end{align*}
Introducing the spherical coordinates again, we get
\begin{align*}
G_{3, 2} \leq& C \| f \|_{L^\infty_\a(\O \times \R^3)} |x - y| \int_0^\infty \int_\O \left| k \left( \z, \rho \frac{y - z}{|y - z|} \right) \right| \frac{\rho}{|y - z|^3}  e^{- \a \rho^2} \,dz d\rho\\
\leq& C \| f \|_{L^\infty_\a(\O \times \R^3)} |x - y|\\
&\times \int_0^\infty \int_{|x - y|}^{\diam(\O)} \int_0^\pi \int_0^{2\pi} \left| k \left( \z, \rho \hat{\z}_* \right) \right| \frac{\rho}{r}   e^{- \a \rho^2} \sin \theta \,d\phi d\theta dr d\rho\\
=& C \| f \|_{L^\infty_\a(\O \times \R^3)} |x - y| \left( \int_{|x - y|}^{\diam(\O)} \frac{1}{r}\,dr \right) \int_{\R^3} \frac{1}{|\z_*|} \left| k \left( \z, \z_* \right) \right| e^{-\a |\z_*|^2} \,d\z_*\\
\leq& C \| f \|_{L^\infty_\a(\O \times \R^3)} |x - y| \left( \left| \log |x - y| \right| + 1 \right) e^{-\a |\z|^2}.
\end{align*}

For the $G_{4, 2}$ term, since
\[
\left| \frac{1}{|x - z|^2} - \frac{1}{|y - z|^2}  \right| \leq \frac{(|x - z| + |y - z|) ||x - z| - |y - z||}{|x - z|^2 |y - z|^2} \leq C \frac{|x - y|}{|y - z|^3}
\]
for $z \in \O_2$, we employ the same argument as for the estimate of $G_{3, 2}$ to obtain
\begin{align*}
G_{4, 2} \leq C \| f \|_{L^\infty_\a(\O \times \R^3)} |x - y| \left( \left| \log |x - y| \right| + 1 \right)e^{-\a |\z|^2} .
\end{align*}
This completes the proof of Proposition~\ref{prop:Hol_G}.

\end{proof}

Since $|x - y| |\log |x - y| | \leq C_\epsilon |x - y|^{1 - \epsilon}$ for $0 < \epsilon < 1$, Proposition \ref{prop:Hol_G} implies that $G(x, \z) \in \cC^\sigma_\a(\O \times \R^3)$ for all $0 < \sigma < 1$. By Lemma \ref{lem:S_Hol} and Lemma \ref{thm:well-posed_lin}, we obtain the following H\"older estimate.

\begin{lemma} \label{lem:Hol_III}
Let $0 \leq \a < 1/2$. Suppose $f_0 \in L^\infty_\a(\Gamma^-)$ and  $\phi \in L^\infty_{\a, 1}(\O \times \R^3)$. Then, for any $1/2 < \sigma < 1$, we have
\begin{align*}
&|III(x, \z) - III(y, \z)|\\ 
\leq& C_\sigma \left( \| f_0 \|_{L^\infty_\a(\Gamma^-)}+ \| \phi \|_{ L^\infty_{\a, 1}(\O \times \R^3) } \right) \left( N(x, y, \z)^{-1} + d_{x, y}^{\sigma - 1} \right) \frac{1 + |\z|}{|\z|^\sigma} |x - y|^\sigma e^{-\a |\z|^2}
\end{align*}
for all $x, y \in \Omega$ and $\z \in \R^3 \setminus \{ 0 \}$.
\end{lemma}

\section{A generalized Fredholm alternative theorem}

 On Page 276 of \cite{CP}, Cercinani and Palczewski used a Fredholm alternative argument to establish the existence theory of a weakly nonlinear slab problem of stationary Boltzmann equation. They used Theorem~\ref{Fredholm-powerCompact} without giving a precise statement nor quoting a theorem. For the self-containedness, in this section,  following the idea of the proof of the Fredholm alternative theorem for identity-compact operators, we address the proof of Theorem \ref{Fredholm-powerCompact} in detail.  We first prove the following lower bound lemma for power compact operators. 

\begin{lemma}\label{lower_bound}
Let $T$ be a bounded linear operator on a Banach space $X$. Suppose $T$ is power compact. Let $\lambda$ be a non-zero complex number which is not an eigenvalue of $T$. Then there exists a positive constant $c$ such that 
\begin{equation} \label{lower1}
\| (T - \lambda) x \| \ge c \| x\|
\end{equation}
for all $x \in X$.
\end{lemma}
\begin{proof}
We shall proceed a proof by contradiction argument to verify Lemma~\ref{lower_bound}. Suppose \eqref{lower1} does not hold; then there exists a sequence $\{x_n\} \subset X$ with $\|x_n\| =1$ for $n=1,2,3, \cdots,$ such that 
\begin{equation}
\label{lower2}
\lim_{n \to \infty} (T - \lambda) x_n = 0 \in X.
\end{equation} 
Due to the compactness of $T^k$, there exists   a subsequence  $ \{ x_{n_j}\} \subset \{x_n\}$ and an element $y \in X$   such that 
\begin{equation}
\label{lower3}
\lim_{n_j \to \infty} T^k x_{n_j} = y.
\end{equation}
Applying $T^{k-1}$ to \eqref{lower2} and taking \eqref{lower3} into consideration, since $\lambda \ne 0$, we see that
\begin{equation}
\label{lower4}
\lim_{n_j \to \infty} T^{k-1} x_{n_j} = \lambda^{-1} y.
\end{equation}
 Similarly, applying $T^{k-2}$ to \eqref{lower2} and taking \eqref{lower4} into consideration, we obtain
\[
\lim_{n_j \to \infty} T^{k-2} x_{n_j} = \lambda^{-2} y.
\] 
Inductively, we come up with
\[
\lim_{n_j \to \infty} x_{n_j} = \lambda^{-k} y,
\]  
which implies $y \ne 0$ since $\| x_{n_j}\| =1$ and $\lambda \ne 0$. Moreover, we have
\[
0 = \lim_{n_j \to 0}  (T - \lambda) x_{n_j}  = \lambda^{-k} (T-\lambda)y,
\]
which violates that $\lambda$ is not an eigenvalue of $T$. The proof of Lemma \ref{lower_bound} is complete.
 \end{proof}
 
Lemma \ref{lower_bound} implies that $\mathcal N(T-\lambda) =0$. Hence,  the operator $T-\lambda$ is injective. Lemma \ref{lower_bound} also implies that  and if the inverse operator $(T-\lambda)^{-1}$ exists, $(T-\lambda)^{-1}$ is a bounded operator. Thus, in order to complete the proof of Theorem \ref{Fredholm-powerCompact}, we only need to show that the operator $T-\lambda$ is surjective.  
  
Next, we shall prove that the range of $T -\lambda$, denoted by $\mathcal R(T-\lambda)$, is closed in $X$.
   
\begin{lemma} \label{closed}
 Under the assumptions of Theorem \ref{Fredholm-powerCompact}, the range of the operator $T-\lambda$ is closed in $X$.
 \end{lemma}

\begin{proof}
To prove the lemma, for any sequence  $\{x_n\}_{n=1}^{\infty}$ in $X$ such that
\begin{equation}
\label{wequation}
\lim_{n \to \infty} (T- \lambda) x_n = w,
\end{equation}
we claim that there exists $x \in X$ such that $(T-\lambda)x = w$. By Lemma \ref{lower_bound} and \eqref{wequation}, we see that $\{x_n\}_{n=1}^{\infty}$ is a Cauchy sequence in $X$. 

We shall show that the limit $x = \lim_{n \to \infty} x_n$ satisfies $(T-\lambda x) = w$. Indeed, we have
\begin{align*}
\| w - (T - \lambda)x \| \leq& \| w - (T - \lambda) x_n \| + \| (T - \lambda)(x_n - x) \|\\
\leq& \| w - (T - \lambda) x_n \| + \| T - \lambda \| \| x_n - x \|.
\end{align*}
By the fact $\lim_{n \to \infty} x_n = x$ and \eqref{wequation}, both terms on the right hand side tend to $0$ as $n \to \infty$. This completes the proof.
\end{proof}

In order to prove the surjectivity of the operator $T-\lambda$, besides Lemma \ref{closed}, we also need the following celebrated Riesz lemma.
 
\begin{lemma}[Riesz Lemma \cite{K}] \label{Riesz}
Let $W$ be a proper closed subspace of a Banach space $X$. Then, for any given $0< \epsilon <1$, there exists a unit vector $x \in X$ such that 
\[
\dist(x, W) > 1-\epsilon.
\]    
\end{lemma} 
   
We are now in a position to prove that the operator $T-\lambda$ is surjective.
    
\begin{lemma}
Under the assumptions of Theorem \ref{Fredholm-powerCompact}, the operator $T-\lambda$ is surjective.
 \end{lemma} 
 
 \begin{proof}
 Suppose the lemma does not hold. By Lemma \ref{closed},  $\mathcal R(T-\lambda)$ is a proper closed space of $X$.  Define $X_m = \mathcal R((T-\lambda)^m)$ for $m \in \mathbb N$.
 By  Lemma~\ref{lower_bound}, we see $\mathcal N (T-\lambda) = 0$.  We inductively obtain that $X_{m+1}$ is a proper closed subspace of  $X_m$ for $m \in \mathbb N$ by Lemma \ref{closed}. Applying Lemma \ref{Riesz}, there exists a sequence $\{ x_m \} \subset X$ such that $x_m \in X_m$ with $\|x_m\| =1$ and $\dist(x_m, X_{m+1}) > \frac{1}{2}$. Let $ m_1 > m_2 $ be two positive integers. Noting that, by the binomial expansion theorem,  
\[ 
T^k = \lambda^k+ \sum_{j=0}^{k-1} \binom{k}{j} \lambda^j (T-\lambda)^{k-j}.
\]
Hence,
\begin{equation}
\label{surjective1}
 \begin{aligned}
 T^k x_{m_1} - T^k x_{m_2}   = & \lambda^k x_{m_1}  +\Big( \sum_{j=0}^{k-1} \binom{k}{j} \lambda^j (T-\lambda)^{k-j} x_{m_1} \\
     &- \lambda^k x_{m_2} - \sum_{j=0}^{k-1} \binom{k}{j} \lambda^j (T-\lambda)^{k-j} x_{m_2} \Big).
 \end{aligned}
\end{equation}
We notice that the first term on the right hand side of \eqref{surjective1} belongs to $X_{m_1}$ and the remaining terms belong to $X_{m_1 + 1}$. Therefore, we have
\[
 \| T^k x_{m_1} - T^k x_{m_2} \| \ge \frac{1}{2 |\lambda|^k},
\]
which violates that $T^k$ is a compact operator.  This implies that $R(T-\lambda)= X$ and the proof is complete. 
\end{proof}


\section{Linear existence theory}

In this section, we give a detailed account on the idea of employing Theorem~\ref{Fredholm-powerCompact} to establish the existence theory of the inhomogeneous linear stationary Boltzmann equations on a general convex domain. Namely, we shall prove Theorem~\ref{thm:well-posed_lin}.

In \eqref{IE}, the inhomogeneous linear Boltzmann equation \eqref{BVP_lin} is expressed in the integral form
\begin{equation*} 
f = Jf_0 + S_\Omega \phi + S_\Omega Kf,
\end{equation*}
which is equivalent to 
\begin{equation}
\label{IdCmp}
(Id -  S_\Omega K) f = Jf_0 + S_\Omega \phi.
\end{equation}
By Proposition~\ref{prop:bound_J_Calpha} and Proposition~\ref{prop:S_decay}, for  $f_0 \in L^\infty_\a(\Gamma^-)$ and $\phi \in L^\infty_{\a, 1}(\O \times \R^3)$, we see that  
$$
Jf_0,  \,\, S_\Omega \phi  \in  L^\infty_{\a}(\O \times \R^3).
$$
To proceed the proof of Theorem~\ref{thm:well-posed_lin}, we shall apply Theorem~\ref{Fredholm-powerCompact} to the operator $Id -  S_\Omega K$  in the Banach space $X = L^\infty_{\a}(\O \times \R^3)$. To accomplish this, we shall show that 
\begin{enumerate}
\item[(A)]  the operator $S_\Omega K$ is power compact in $X =L^\infty_{\a}(\O \times \R^3)$, and  
\item[(B)]  the operator $Id -  S_\Omega K : X \to X$ is injective. 
\end{enumerate}
{\sc The proof of (a).}
To see (A), we prove the following lemma.
\begin{lemma} \label{cor:compact2}
The operator $S_\O K S_\O K$ is compact on $L^\infty_\a(\O \times \R^3)$. 
\end{lemma}
\begin{proof}
According to Corollary \ref{cor:bound_S_Calpha} and the estimate $\| h \|_{L^\infty_{\a, 1}(\O \times \R^3)} \leq \| h \|_{L^\infty_\a(\O \times \R^3)}$ for $h \in X$, $S_\O : X \to X$ is a bounded linear operator. To prove the operator $S_\O K S_\O K$ is compact in $X$, we only need to show that  $K S_\O K$  is compact in $X$.

 Let $\overline{B}_R$ be the closed ball in $\R^3$  centered at the origin with radius $R>0$. For a measurable set $A$ in $\R^3$, let $\chi_A(\z)$ denote its characteristic function. Set  $\chi_R(\z) = \chi_{\overline{B}_R}(\z)$ be the characteristic function of $\overline{B}_R$ and
$$
T_R := \chi_R(\z) K S_\O K.
$$ 

{\sc Step 1.}
We first prove that, for any $R>0$,  the operator $T_R: X \to X $ is  compact. Let $\{ f_n(x, \z) \}$ be a bounded sequence in $L^\infty_\a(\O \times \R^3)$. Since $\chi_R$, $S_\O$ and $K$ are bounded linear operators in $X$, the sequence $\{ T_R f_n(x, \z) \}$ is a bounded sequence in $X$. In particular, it is bounded in $L^\infty(\O \times \overline{B}_R)$. 

By Proposition \ref{prop:Hol_G} and  Corollary \ref{cor:est_Kdv_alpha_op}, we have
\begin{equation} \label{est:Hol_Gn}
\left\{
\begin{aligned}
& |T_R f_n(x, \z) - T_R f_n(y, \z)| \leq C_\epsilon \| f_n \|_{L^\infty_\a(\O \times \R^3)} |x - y|^{1 - \epsilon}, \\
& |T_R f_n(x, \z_1) - T_R f_n(x, \z_2)| \leq   C \| f_n \|_{L^\infty_\a(\O \times \R^3)} |\z_1 - \z_2|,
\end{aligned}
\right.
\end{equation}
for all $x, y \in \O$ and $\z, \z_1, \z_2 \in \overline{B}_R$. We notice that, by Proposition~\ref{prop:Hol_G}, the function $T_R f_n$ can be be continuously extended up to $\overline{\O \times B_R}$ and the estimate \eqref{est:Hol_Gn} holds for all $x, y \in \overline{\O}$ and $\z, \z_1, \z_2 \in \overline{B}_R$. Since $\{ f_n \}$ is uniformly bounded in $X$, the  estimate \eqref{est:Hol_Gn} implies that $\{ T_R f_n \}$ is equi-continuous with respect to $(x, \z) \in \overline{\O \times B_R}$. Therefore, by the Ascoli-Arzel\'a theorem, we conclude that there exists a subsequence $\{ T_R f_{n_k} \}$ of $\{ T_R f_n \}$ and a continuous function $g_R(x, \z) \in X$ such that $\{ T_R f_{n_k} \}$ converges to $g_R(x, \z)$ uniformly on $\overline{\O \times B_R}$ as $n_k \to \infty$.  This proves that the operator $T_R: X \to X$ is a compact operator.

{\sc Step 2.} To complete the proof, we shall show that $T_R \to  K S_\O K$ as $R \to \infty$ in the sense that the operator norm of   $(K S_\O K- T_R)$ tends to zero as $R \to \infty$.
  By Lemma \ref{lem:est_k_decay}, for any $h \in X$, we see that
\begin{align*}
|(Id - \chi_R(\z)) K S_\O K h(x, \z)| e^{\a |\z|^2} \leq& \chi_{\{ |\z| \geq R \}} e^{\a |\z|^2} \int_{\R^3} |k(\z, \z_*)| |(S_\O Kh)(x, \z_*)|\,d\z_*\\
\leq& \frac{C \chi_{\{ |\z| \geq R \}}}{1 + |\z|} \| S_\O K h \|_{L^\infty_\a(\O \times \R^3)}\\
\leq& \frac{C}{1 + R} \| h \|_{L^\infty_\a(\O \times \R^3)},
\end{align*}
which implies that the operator norm of  $(Id - \chi_R(\z)) KS_\O K : X \to X$ satisfies
\[ 
\| (Id - \chi_R(\z)) KS_\O K \| \leq \frac{C}{1 + R}.
\]
This completes the proof.

To prove (B), we need the following lemma.

\begin{lemma} \label{lem:injective_L2}
The oprerator $Id - S_\O K: L^2(\O \times \R^3) \to L^2(\O \times \R^3)$ is injective. 
\end{lemma} 

\begin{proof}
Let $f \in L^2(\O \times \R^3)$ be a solution of $(Id - S_\O K) f = 0$, i.e.,  $f \in L^2(\O \times \R^3)$ solves the homogeneous boundary value problem:
\[
\begin{cases}
\z \cdot \nabla_x f + \nu f = Kf &\mbox{ in } \O \times \R^3, \\
f = 0&\mbox{ on } \Gamma_-.
\end{cases}
\]
To prove Lemma~\ref{lem:injective_L2}, we need to show $f \equiv 0$.
Multiplying the above equation by $f$ and integrating the resultant over $\O \times \R^3$,  we have
\[
  \int_{\O \times \R^3} (Lf)f\,dxd\z  = \frac{1}{2} \int_{\Gamma_+} f(z, \z)^2 |n(z) \cdot \z|\,d\sigma_z d\z  \ge 0. 
\]

On the other hand, let $P$ be the projection operator from $L^2(\O \times \R^3)$ to $\mathcal{N}(L)$, where $\mathcal{N}(L)$ is the kernel of the linear operator $L$. By the following coercivity estimate, see \cite{Glassey, Grad, Grad1, Mo},   
\[
- \int_{\O \times \R^3} (Lf)f\,dxd\z \geq c_0 \| (Id - P) f \|_{L^2(\O \times \R^3)}^2, \,\, c_0>0,
\]
 we conclude that $(Id - P) f = 0$, which means that $f$ satisfies $L f = 0$. In terms of the original boundary value problem, this implies that $f$ solves the homogeneous boundary value problem of the free transport equation:
\[
\begin{cases}
\z \cdot \nabla_x f = 0 &\mbox{ in } \O \times \R^3, \\
f = 0&\mbox{ on } \Gamma_-.
\end{cases}
\]
Thus, the method of characteristic lines shows that $f \equiv 0$. This completes the proof.
\end{proof}

{\sc The Proof of (B).}
We are now in a position to show (B). For $0 < \a < 1/2$, we see that $L^\infty_\a(\O \times \R^3) \subset L^2(\O \times \R^3)$. Thus, Lemma \ref{lem:injective_L2} implies that the operator $Id - S_\O K: L^\infty_\a(\O \times \R^3) \to L^\infty_\a(\O \times \R^3)$ is injective. 

On the other hand, for the case where $\a = 0$, we do not have such an inclusion. However, we can show that a function $f \in L^\infty(\O \times \R^3)$ satisfying $(Id - S_\O K)f = 0$ necessarily belongs to $L^2(\O \times \R^3)$, which again implies that $f = 0$. By the estimate \eqref{est:S_decay} and Lemma \ref{lem:est_k_decay}, we have
\[
(1 + |\z|^2) |f(x, \z)| = (1 + |\z|^2) |S_\O K f(x, \z)| \leq C \| f \|_{L^\infty(\O \times \R^3)},
\]
or
\[
|f(x, \z)| \leq \frac{C \| f \|_{L^\infty(\O \times \R^3)}}{1 + |\z|^2}
\] 
with some positive constant $C$. Since the domain $\O$ is bounded, the above estimate implies that $f \in L^2(\O \times \R^3)$. This completes the proof of (B).

Finally, by applying Theorem~\ref{Fredholm-powerCompact} to the equation \eqref{IdCmp}, we obtain Theorem~\ref{thm:well-posed_lin}.
\end{proof}

\section{Differentiability} \label{sec:diff}

We shall prove Theorem~\ref{lem:sol_dz_lin} in this section. Since   $f_0 \in \tilde{\cB}_\a \subset \cB_\a \subset  L^\infty_\a(\Gamma^-)$ and $\phi \in \tilde{\cC}^1_\a \subset   \cC^1_\a \subset   L^\infty_{\a, 1}(\O \times \R^3)$, to prove Theorem~\ref{lem:sol_dz_lin}, we only need to prove the regularity   of the solution obtained in Theorem~\ref{thm:well-posed_lin} and show that the solution satisfies \eqref{ineq:sol_dz_lin}. To do this, we consider
\[
f(x, \z) = I(x, \z) + II(x, \z) + III(x, \z),
\]
where  $I(x, \z)$, $II(x, \z)$ and $III(x, \z)$ are as defined by \eqref{I}-\eqref{III}.

\subsection{Estimate of $I$ and $II$} \label{subsec:I_II_dx}

First, by applying  Lemma \ref{lem:Hol_I} with $\sigma =1$, we obtain the following estimate.

\begin{lemma} \label{lem:Idx}
Let $f_0 \in \cB_\a$ and $\phi \in \cC^1_\a$. Then, we have
\[
|\nabla_x I(x, \z)| \leq C \left( \| f_0 \|_{\cB_\a} + \| \phi \|_{\cC^1_\a} \right) w(x, \z)^{-1} e^{-\a |\z|^2}
\]
for a.e. $(x, \z) \in \O \times \R^3$.
\end{lemma}

Applying Lemma \ref{lem:Hol_II} with $\sigma =1$, we obtain the following estimate.

\begin{lemma} \label{lem:IIdx}
Let $f_0 \in \cB_\a$ and $\phi \in \cC^1_\a$. Then, we have
\begin{align*}
|\nabla_x II(x, \z)| \leq C \left( \| f_0 \|_{\cB_\a} + \| \phi \|_{\cC^1_\a} \right) w(x, \z)^{-1} e^{-\a |\z|^2}
\end{align*}
for a.e. $(x, \z) \in \O \times \R^3$.
\end{lemma}

\subsection{Estimate of $III$} \label{subsec:IIIdx}

 Recall that $III = S_\O G$, where $G = K S_\O K f$ as we introduced in Subsection \ref{subsec:Hol_III}. In order to give an estimate for the $x$ derivative of $III$, we shall discuss the differentiability of the function $G$.
 
 First, combining Lemma \ref{lem:Hol_I}, Lemma \ref{lem:Hol_II}, and Lemma \ref{lem:Hol_III}, we obtain the following lemma.

\begin{lemma} \label{lem:sol_Hol_lin}
Let $\O$ be an open bounded convex domain in $\R^3$ with $C^2$ boundary of positive Gaussian curvature. Fix any $1/2 < \sigma < 1$ and $0 \leq \a < 1/2$.  There exists a positive constant $C$ such that : for any  $f_0 \in \cB_\a$ and $\phi \in \cC^\sigma_\a$,  the solution  $f \in L^\infty_\a(\O \times \R^3)$ to the integral equation \eqref{IE} satisfies
\begin{equation} \label{ineq:HF}
\begin{split}
|f(x, \z) - f(y, \z)| \leq& C \left( \| f_0 \|_{\cB_\a} + \| \phi \|_{\cC^\sigma_\a} \right)\\ 
&\times \left( N(x, y, \z)^{-1} + d_{x, y}^{\sigma - 1} \right) \frac{1 + |\z|}{|\z|^\sigma} |x - y|^{\sigma} e^{-\a |\z|^2}
\end{split}
\end{equation}
 for all $(x, y, \z) \in \overline{\Omega} \times \overline{\Omega} \times (\R^3 \setminus \{ 0 \})$. 
\end{lemma}

For the gradient estimate of  of the function $G(x, \z)$, we start with the following lemma, see also  Lemma 10.5 of  \cite{CHK}. Note that the setting of the function space in this paper is different from that in \cite{CHK}. But, for this gradient estimate, the method of the proof is basically the same.

\begin{lemma} \label{lem:G_dx}
Let $f \in L^\infty_\a(\O \times \R^3)$ be the solution to the integral equation \eqref{IE}. Suppose that the solution $f$ satisfies the estimate \eqref{ineq:HF} for all $(x, \z) \in \O \times (\R^3 \setminus \{ 0 \})$ and $1/2 < \sigma < 1$. Then, $G$ is differentiable with respect to $x \in \O$ for all $(x, \z) \in \O \times (\R^3 \setminus \{ 0 \})$. Moreover, we have
\[
\left| \nabla_x G(x, \z) \right| \leq C_{\a, \sigma} \left( \| f_0 \|_{\cB_\a} + \| \phi \|_{\cC^1_\a} \right) \left( |\log d_x| + 1 \right) e^{-\a |\z|^2 }.
\]
\end{lemma}

\begin{proof}
We recall that
\begin{align*}
&G(x, \z)\\ 
=& \int_0^\infty \int_\O \int_{\R^3} k \left( \z, \rho \frac{x - z}{|x - z|} \right) e^{-\nu(\rho) \frac{|x - z|}{\rho}} k \left( \rho \frac{x - z}{|x - z|}, \z_* \right) f(z, \z_*) \frac{\rho}{|x - z|^2}\,d\z_*dzd\rho.
\end{align*}

Take $x_0 \in \O$ and $0 < \delta < d_{x_0}/2$. Also, let $\O_\delta := \O \setminus B(x_0, \delta)$ and define  
\begin{align*}
&G^\delta(x, \z)\\ 
=& \int_0^\infty \int_{\O_\delta} \int_{\R^3} k \left( \z, \rho \frac{x - z}{|x - z|} \right) e^{-\nu(\rho) \frac{|x - z|}{\rho}} k \left( \rho \frac{x - z}{|x - z|}, \z_* \right) f(z, \z_*) \frac{\rho}{|x - z|^2}\,d\z_*dzd\rho.
\end{align*}
It is clear that the function $G^\delta$ is differentiable at $x_0$. We shall show that $\nabla_x G^\delta(x_0, \z)$ converges as $\delta \downarrow 0$.

Let $\O_1 := B(x_0, d_{x_0}/2)$ and $\O_2 := \O \setminus B(x_0, d_{x_0}/2)$. Also, $G_i^\delta$, $i = 1, 2$, denote the corresponding integrals. 

We notice that 
\begin{align*}
&\left| \nabla_x \left( k \left( \z, \rho \frac{x - z}{|x - z|} \right) e^{-\nu(\rho) \frac{|x - z|}{\rho}} k \left( \rho \frac{x - z}{|x - z|}, \z_* \right) \frac{\rho}{|x - z|^2} \right) \right|\\
\leq& C \left| (\nabla_{\z_*} k) \left( \z, \rho \frac{x - z}{|x - z|} \right) \right| \left| k \left( \rho \frac{x - z}{|x - z|}, \z_* \right) \right| \frac{\rho^2}{|x - z|^3}\\
&+ C \left| k \left( \z, \rho \frac{x - z}{|x - z|} \right) \right| \left| k \left( \rho \frac{x - z}{|x - z|}, \z_* \right) \right| \frac{\rho}{|x - z|^3}\\
&+ C \left| k \left( \z, \rho \frac{x - z}{|x - z|} \right) \right| \left| (\nabla_\z k) \left( \rho \frac{x - z}{|x - z|}, \z_* \right) \right| \frac{\rho^2}{|x - z|^3}.
\end{align*}
For $\delta < d_{x_0}/2$, we see that $\Omega_{\delta} \cap \Omega_2 = \Omega_2$. By Lemma \ref{lem:est_k_decay}, we have
\begin{align*}
&\int_0^{\infty} \int_{\Omega_{\delta} \cap  \Omega_2} \int_{\mathbb R^3} \left| (\nabla_{\z_*} k) \left( \z, \rho \frac{x_0 - z}{|x_0 - z|} \right) \right|\\
&\times \left| k \left( \rho \frac{x_0 - z}{|x_0 - z|}, \z_* \right) \right| \frac{\rho^2}{|x_0 - z|^3}|f(z, \z_*)| d \z_* dz d \rho\\
\leq& C \| f \|_{L^\infty_\a(\O \times \R^3)} \int_0^{\infty} \int_{\Omega_2}  \left| (\nabla_{\z_*} k) \left( \z, \rho \frac{x_0 - z}{|x_0 - z|} \right) \right| \frac{\rho^2}{|x_0 - z|^3} e^{-\alpha \rho^2}  dz d \rho.
\end{align*}
By introducing the spherical coordinates $z = x_0 + r \omega$ with $d_{x_0}/2 < r \le \diam(\O)$ and $\omega \in S^2$, as we did in the proof of Proposition \ref{prop:Hol_G}, we have
\begin{align*}
&\int_{\Omega_2}  \left| (\nabla_{\z_*} k) \left( \z, \rho \frac{x_0 - z}{|x_0 - z|} \right) \right| \frac{\rho^2}{|x_0 - z|^3} e^{-\alpha \rho^2}  dz d \rho\\
\leq& \int_0^{\infty} \int_{S^2}\int_{\frac{d_{x_0}}{2}}^{\diam(\O)}  \left| (\nabla_{\z_*} k) \left( \z, \rho \omega \right) \right|  e^{-\alpha \rho^2} \rho^2 \frac{1}{r} drd\omega d \rho.
\end{align*}
Furthermore, by letting $\z_* = \rho \omega$ and Lemma \ref{lem:est_Kdv_alpha}, we have
\begin{align*}
&\int_0^{\infty} \int_{S^2}\int_{\frac{d_{x_0}}{2}}^{\diam(\O)}  \left| (\nabla_{\z_*} k) \left( \z, \rho \omega \right) \right|  e^{-\alpha \rho^2} \rho^2 \frac{1}{r} drd\omega d \rho\\
\le& C(|\log d_{x_0}| + 1) \int_{\mathbb R^3}   \left| (\nabla_{\z_*} k) \left( \z, \z_* \right) \right|  e^{-\alpha |\z_*|^2} d \z_*\\
\le& C(|\log d_{x_0}| + 1) e^{- \a |\z|^2}.
\end{align*}
Since the other terms can be estimated in the same fashion, we obtain 
\[
|\nabla_x G_2 (x_0, \z)| = \left| \lim_{\delta \downarrow 0} \nabla_x G_2^\delta(x_0, \z) \right|  \leq C \| f \|_{L^\infty_\a(\O \times \R^3)} (|\log d_{x_0}| + 1) e^{- \a |\z|^2}
\]
for all  $(x_0, \z) \in \O \times (\R^3 \setminus \{ 0 \})$.

We decompose the derivative of the first term into two parts:
\[
\nabla_x G_1^\delta (x, \z) = G_{1, 1}^\delta(x, \z) + G_{1, 2}^\delta(x, \z),
\]
where
\begin{align*}
G_{1, 1}^\delta(x, \z) :=& \int_0^\infty \int_{\O_\delta \cap \O_1} \int_{\R^3} \nabla_x \left( k \left( \z, \rho \frac{x - z}{|x - z|} \right) e^{-\nu(\rho) \frac{|x - z|}{\rho}} k \left( \rho \frac{x - z}{|x - z|}, \z_* \right) \right.\\ 
&\left. \times \frac{\rho}{|x - z|^2} \right) \left( f(z, \z_*) - f(x, \z_*) \right)\,d\z_*dzd\rho
\end{align*}
and
\begin{align*}
G_{1, 2}^\delta(x, \z) :=& \int_0^\infty \int_{\O_\delta \cap \O_1} \int_{\R^3} \nabla_x \left( k \left( \z, \rho \frac{x - z}{|x - z|} \right) e^{-\nu(\rho) \frac{|x - z|}{\rho}} k \left( \rho \frac{x - z}{|x - z|}, \z_* \right) \right.\\
&\times \left. \frac{\rho}{|x - z|^2} \right) f(x, \z_*) \,d\z_*dzd\rho.
\end{align*}

For  $G_{1, 1}^\delta$, we use the estimate \eqref{ineq:HF} and the estimate (see Proposition~\ref{prop:N1/2}) 
\[
N(x_0, z, \z)^{-1} + d_{x_0, z}^{\sigma - 1} \leq C_\sigma d_{x_0}^{-\frac{1}{2}}
\]
to obtain
\begin{align*}
&\int_0^\infty \int_{\O_\delta \cap \O_1} \int_{\R^3} \left| \nabla_x \left( k \left( \z, \rho \frac{x_0 - z}{|x_0 - z|} \right) e^{-\nu(\rho) \frac{|x_0 - z|}{\rho}} k \left( \rho \frac{x_0 - z}{|x_0 - z|}, \z_* \right) \frac{\rho}{|x_0 - z|^2} \right) \right|\\
&\quad \times \left| f(z, \z_*) - f(x_0, \z_*) \right|\,d\z_*dzd\rho\\
\leq& C_\sigma d_{x_0}^{-\frac{1}{2}} \left( \| f_0 \|_{\cB_\a} + \| \phi \|_{\cC^1_\a} \right) \sum_{j = 1}^3 K_{j, \delta},
\end{align*}
where
\begin{align*}
K_{1, \delta} :=& \int_0^\infty \int_{\O_\delta \cap \O_1} \int_{\R^3} \left| (\nabla_{\z_*} k) \left( \z, \rho \frac{x_0 - z}{|x_0 - z|} \right) \right| \left| k \left( \rho \frac{x_0 - z}{|x_0 - z|}, \z_* \right) \right|\\
&\times \frac{\rho^2}{|x_0 - z|^{3 - \sigma}} \frac{1 + |\z_*|}{|\z_*|^\sigma} e^{- \a |\z_*|^2}\,d\z_* dz d\rho,\\
K_{2, \delta} :=& \int_0^\infty \int_{\O_\delta \cap \O_1} \int_{\R^3} \left| k \left( \z, \rho \frac{x_0 - z}{|x_0 - z|} \right) \right| \left| k \left( \rho \frac{x_0 - z}{|x_0 - z|}, \z_* \right) \right|\\
&\times \frac{\rho}{|x_0 - z|^{3 - \sigma}} \frac{1 + |\z_*|}{|\z_*|^\sigma} e^{- \a |\z_*|^2}\,d\z_* dz d\rho,\\
K_{3, \delta} :=& \int_0^\infty \int_{\O_\delta \cap \O_1} \int_{\R^3} \left| k \left( \z, \rho \frac{x_0 - z}{|x_0 - z|} \right) \right| \left| (\nabla_\z k) \left( \rho \frac{x_0 - z}{|x_0 - z|}, \z_* \right) \right|\\
&\times \frac{\rho^2}{|x_0 - z|^{3 - \sigma}} \frac{1 + |\z_*|}{|\z_*|^\sigma} e^{- \a |\z_*|^2}\,d\z_* dz d\rho.
\end{align*}
For $K_{3, \delta}$ with $\delta < d_{x_0}/2$, we apply Lemma \ref{lem:est_Kdv_alpha} and the same change of variables for integrations as for the estimate of $\nabla_x G_2$ to obtain
\begin{align*}
K_{3, \delta} \leq& C \int_0^\infty \int_{\O_1} \left| k \left( \z, \rho \frac{x_0 - z}{|x_0 - z|} \right) \right| \frac{\rho^2(1 + \rho)}{|x_0 - z|^{3 - \sigma}} e^{-\a \rho^2} \,dz d\rho\\
\leq& C \left( \int_0^{d_{x_0}} r^{\sigma-1}\,dr \right) \int_{\R^3} |k(\z, \z_*)| (1+|\z_*|) e^{-\a |\z_*|^2}\,d\z_*\\
\leq& C d_{x_0}^\sigma e^{-\a |\z|^2}.
\end{align*}
Here, we have used the estimate \eqref{est:k1}. The other terms can be estimated in the same fashion to obtain
\[
\sum_{j = 1}^3 K_{j, \delta} \leq C d_{x_0}^{\sigma} e^{- \a |\z|^2},
\]
where the constant $C$ is independent of $\delta$. Thus, the limit $G_{1, 1}(x_0, \z) := \lim_{\delta \downarrow 0} G_{1, 1}^\delta(x_0, \z)$ exists and
\begin{align*}
|G_{1, 1}(x_0, \z)| \leq& C_{\a, \sigma} d_{x_0}^{\sigma -\frac{1}{2}} \left( \| f_0 \|_{\cB_\a} + \| \phi \|_{\cC^1_\a} \right) e^{- \a |\z|^2}\\
\leq& C_{\a, \sigma} \left( \| f_0 \|_{\cB_\a} + \| \phi \|_{\cC^1_\a} \right) e^{- \a |\z|^2}.
\end{align*}
Here, we have used the assumption $1/2 < \sigma < 1$ to guarantee that $d_{x_0}^{\sigma -\frac{1}{2}} \leq C_\sigma$.

For $G_{1, 2}^\delta$, we notice that
\begin{align*}
G_{1, 2}^\delta(x_0, \z) =& - \int_0^\infty \int_{\O_\delta \cap \O_1} \int_{\R^3} \nabla_z \left( k \left( \z, \rho \frac{x_0 - z}{|x_0 - z|} \right) e^{-\nu(\rho) \frac{|x_0 - z|}{\rho}} \right.\\
&\times \left. k \left( \rho \frac{x_0 - z}{|x_0 - z|}, \z_* \right) \frac{\rho}{|x_0 - z|^2} \right) f(x_0, \z_*) \,d\z_*dzd\rho.
\end{align*}
We perform integration by parts with respect to the $z$ variable and change the variable of integration to obtain
\begin{align*}
G_{1, 2}^\delta(x_0, \z) =& - \int_0^\infty \int_{\R^3} \int_{\p B(x_0, \frac{d_{x_0}}{2})} k \left( \z, \rho \frac{x_0 - z}{|x_0 - z|} \right) e^{-\nu(\rho) \frac{|x_0 - z|}{\rho}} \\
&\times k \left( \rho \frac{x_0 - z}{|x_0 - z|}, \z_* \right) \frac{\rho}{|x_0 - z|^2}f(x_0, \z_*) \frac{z - x_0}{|z - x_0|} \,d\z_* d\sigma_z d\rho\\
&+ \int_0^\infty \int_{\R^3} \int_{\p B(x_0, \delta)} k \left( \z, \rho \frac{x_0 - z}{|x_0 - z|} \right) e^{-\nu(\rho) \frac{|x_0 - z|}{\rho}} \\
&\times k \left( \rho \frac{x_0 - z}{|x_0 - z|}, \z_* \right) \frac{\rho}{|x_0 - z|^2}f(x_0, \z_*)  \frac{z - x_0}{|z - x_0|}  \,d\z_* d\sigma_z d\rho \\
=& \int_{\R^3} \int_{\R^3} k(\z, \z_{**}) \left( - e^{-\frac{\nu(|\z_{**}|)}{|\z_{**}|} \frac{d_{x_0}}{2}} + e^{-\frac{\nu(|\z_{**}|)}{|\z_{**}|} \delta} \right) k(\z_{**}, \z_*)\\ 
&\times f(x_0, \z_*) \frac{\z_{**}}{|\z_{**}|^2}\,d\z_{**} d\z_*.
\end{align*}
The integrand on the right hand side is dominated by
\[
\frac{2}{|\z_{**}|} |k(\z, \z_{**})| |k(\z_{**}, \z_*)| \| f \|_{L^\infty_\a(\O \times \R^3)} e^{- \a |\z_*|^2}.
\]
The dominated convergence theorem implies that the limit $G_{1, 2}(x_0, \z) := \lim_{\delta \downarrow 0} G_{1, 2}^\delta(x_0, \z)$ exists and the above estimate indicates that
\[
|G_{1, 2}(x_0, \z)| \leq C_\a \| f \|_{L^\infty_\a(\O \times \R^3)} e^{- \a |\z|^2}.
\]

Therefore, the limit $\nabla_x G_1(x_0, \z) = \lim_{\delta \downarrow 0} \nabla_x G_1^\delta(x_0 ,\z)$ exists and satisfies
\[
|\nabla_x G_1(x_0, \z)| \leq C_{\a, \sigma} \left( \| f_0 \|_{\cB_\a} + \| \phi \|_{\cC^1_\a} \right) e^{- \a |\z|^2}.
\]
This completes the proof.
\end{proof}

We are now in a position to give an estimate for the $x$ derivative of $III$. The formal differentiation yields
\begin{align*}
\nabla_x III(x, \z) =& e^{-\nu(\z) \tau_-(x, \z)} G(q(x, \z), \z) \nabla_x \tau_-(x, \z)\\
&+ \int_0^{\tau_-(x, \z)} e^{-\nu(\z) t} \nabla_x G(x - t\z, \z)\,dt.
\end{align*}

For the first term, by Corollary \ref{cor:estTP}, we have
\begin{align*}
|e^{-\nu(\z) \tau_-(x, \z)} G(q(x, \z), \z) \nabla_x \tau_-(x, \z)| \leq& C \| G \|_{L^\infty_\a(\O \times \R^3)} \frac{w(x, \z)^{-1}}{1 + |\z|} e^{-\a |\z|^2}\\
\leq& C (\| f_0 \|_{\cB_\a} + \| \phi \|_{\cC^1_\a}) w(x, \z)^{-1} e^{-\a |\z|^2}.
\end{align*}
Here, we have applied Proposition \ref{prop:Hol_G} to estimate the boundary value of $G$ by $\| G \|_{L^\infty_\a(\O \times \R^3)}$. Also, we have used the boundedness of the operator $K S_\O K$ on $L^\infty_\a(\O \times \R^3)$ and the estimate \eqref{bound_Linfty} in Theorem \ref{thm:well-posed_lin}. 

For the second term, we notice that $|\log d_x| + 1 \leq C d_x^{-1/2}$ for some positive constant $C$. Thus, Lemma \ref{lem:S1/2} and Lemma \ref{lem:G_dx} imply that
\begin{align*}
&\left| \int_0^{\tau_-(x, \z)} e^{-\nu(\z) t} \nabla_x G(x - t\z, \z)\,dt \right|\\
\leq& C (\| f_0 \|_{\cB_\a} + \| \phi \|_{\cC^1_\a}) \int_0^{\tau_-(x, \z)} e^{-\nu(\z) t} d_{x - t\z}^{-\frac{1}{2}}\,dt e^{-\a |\z|^2}\\
\leq& \frac{C}{|\z|} (\| f_0 \|_{\cB_\a} + \| \phi \|_{\cC^1_\a}) e^{-\a |\z|^2}\\
\leq& C (\| f_0 \|_{\cB_\a} + \| \phi \|_{\cC^1_\a}) w(x, \z)^{-1} e^{-\a |\z|^2}.
\end{align*}

Therefore, we have
\begin{equation} \label{est:IIIdx}
|\nabla_x III(x, \z)| \leq C (\| f_0 \|_{\cB_\a} + \| \phi \|_{\cC^1_\a}) w(x, \z)^{-1} e^{-\a |\z|^2},
\end{equation}
for all $(x, \z) \in \O \times (\R^3 \setminus \{ 0 \})$.

Combining the estimate \eqref{est:IIIdx}, Lemma \ref{lem:Idx} and Lemma \ref{lem:IIdx}, we obtain the following theorem.

\begin{theorem} \label{thm:sol_dz_linxxxx}
Let $\O$ be an open bounded convex domain with $C^2$ boundary of positive Gaussian curvature, and let $0 \leq \a < 1/2$. Suppose \eqref{assumption_B1} holds. Then, for any $f_0 \in {\cB}_\a$ and $\phi \in {\cC}^1_\a$, the integral equation \eqref{IE} has a unique solution in ${\cW}_\a$ that satisfies 
\[
\| f \|_{{\cW}_\a} \leq C \left( \| f_0 \|_{{\cB}_\a} + \| \phi \|_{{\cC}^1_\a} \right)
\]
for some positive constant $C$ independent of $f_0$ and $\phi$. 
\end{theorem}

\subsection{The proof of Theorem \ref{lem:sol_dz_lin}}

Based on Theorem~\ref{thm:sol_dz_linxxxx}, we are now in a position to prove Theorem \ref{lem:sol_dz_lin}. 

\begin{proof}
Let $f$ be the solution obtained in Theorem~\ref{thm:sol_dz_linxxxx} with  $f_0 \in \tilde{\cB}_\a$ and $\phi \in \tilde{\cC}^1_\a$.
We consider the integral form of the solution
\[
f = Jf_0 + S_\Omega \phi + S_\Omega Kf.
\]
By Proposition \ref{prop:bound_J_Calpha}, Corollary \ref{cor:J_dx} and Lemma \ref{lem:J_dz}, we see that
\begin{equation}
\label{6.3}
\| J f_0 \|_{\tilde{\cW}_\a} \leq C \| f_0 \|_{\tilde{\cB}_\a}.
\end{equation}
Also, by Corollary \ref{cor:bound_S_Calpha}, Corollary \ref{cor:S_dx} and Lemma \ref{lem:S_dz}, we have
\begin{equation}
\label{6.4}
\| S_\O \phi \|_{\tilde{\cW}_\a} \leq C \| \phi \|_{\tilde{\cC}^1_\a}.
\end{equation}
To prove $f \in \tilde{\cW}_\a$, it suffices to show that $S_\O K f \in \tilde{\cW}_\a$. By Corollary \ref{cor:S_dx} and Lemma \ref{lem:S_dz} again, we only need  to prove that $K f \in \tilde{\cC}^1_\a$.

Since the solution obtained in Theorem~\ref{thm:sol_dz_linxxxx} belongs to  ${\cW}_\a$, by Lemma~\ref{lem:est_k_0} and Proposition~\ref{prop:N1/2}, we see that 
\begin{equation} \label{6.5}
 \| Kf\|_{\cC^1_\a} \le C \|f\|_{{\cW}_\a}. 
\end{equation}
Thus, it remains  to show that
\begin{equation} \label{6.6}
|\nabla_\z Kf(x, \z)| \leq C (\| f_0 \|_{\tilde{\cB}_\a} + \| \phi \|_{\tilde{\cC}^1_\a}) d_x^{-\frac{1}{2}} (1 + |\z|) e^{-a |\z|^2}.
\end{equation}
To verify \eqref{6.6}, by Corollary \ref{cor:est_Kdv_alpha_op} and Theorem \ref{thm:well-posed_lin}, we have
\begin{align*}
|\nabla_\z Kf(x, \z)| \leq& C_\a \| f \|_{L^\infty_\a(\O \times \R^3)} e^{-\a |\z|^2}\\
\leq& C (\| f_0 \|_{L^\infty_\a(\Gamma^-)} + \| \phi \|_{L^\infty_{\a, 1}(\O \times \R^3)}) e^{-a |\z|^2}\\
\leq& C (\| f_0 \|_{\tilde{\cB}_\a} + \| \phi \|_{\tilde{\cC}^1_\a}) d_x^{-\frac{1}{2}} (1 + |\z|) e^{-a |\z|^2}.
\end{align*}
 Summing up \eqref{6.3}-\eqref{6.6} and taking  Theorem \ref{thm:sol_dz_linxxxx}, Corollary \ref{cor:bound_S_Calpha}, Corollary \ref{cor:S_dx} and Lemma \ref{lem:S_dz} into consideration, we have obtain the estimate \eqref{ineq:sol_dz_lin}. This completes the proof.
\end{proof}

\section{Regularity for the nonlinear case}

In this section, we divide the proof of  Lemma \ref{lem88} into three parts. 
To see the estimate \eqref{lem91}, we prove Lemma~\ref{lem:Gamma}. For the estimate of \eqref{lem92}, we prove Lemma~\ref{lem:Gamma_dx}. To show the estimate \eqref{lem93}, we prove Lemma~\ref{lem:Gamma_dz}.

In our following proof, we decompose the nonlinear term $\Gamma$ into two parts:
\[
\Gamma(h_1,h_2) = \pi^{-\frac{3}{4}} \left( \Gamma_{\gain}(h_1, h_2) - \Gamma_{\loss}(h_1, h_2) \right),
\]
where
\begin{align*}
\Gamma_{\gain}(h_1, h_2) :=& \int_{\mathbb{R}^{3}}\int_{0}^{2\pi}\int_{0}^{\frac{\pi}{2}}e^{-\frac{|\z_*|^2}{2}} h_1(\z') h_2(\z_*') B(|\z-\z_*|, \theta)\,d\theta \, d\phi\,d\z_*,\\
\Gamma_{\loss}(h_1, h_2) :=& \int_{\mathbb{R}^{3}}\int_{0}^{2\pi}\int_{0}^{\frac{\pi}{2}}e^{-\frac{|\z_*|^2}{2}} h_1(\z) h_2(\z_*) B(|\z-\z_*|, \theta)\,d\theta \, d\phi\,d\z_*
\end{align*}
As adopted by many authors, we call $\Gamma_{\gain}$ and $\Gamma_{\loss}$ the gain term and the loss term, respectively.

We start with the weighted $L^\infty$ estimate for the nonlinear term.

\begin{lemma} \label{lem:Gamma}
Fix $\a \geq 0$ and suppose \eqref{assumption_B1} holds. Then, there exists some positive constant $C$ such that
for any $h_1, h_2 \in L^{\infty}_{\alpha}(\O \times \R^3)$, we have
\[
|\Gamma(h_1, h_2)(x, \z)| \leq C \| h_1 \|_{L^\infty_\a(\O \times \R^3)} \| h_2 \|_{L^\infty_\a(\O \times \R^3)} e^{-\alpha |\z|^2} (1 + |\z|)^\gamma
\]
for a.e. $(x, \z) \in \Omega \times \R^3$.
\end{lemma}

\begin{proof}
For the gain term, invoking the relation $|\z|^2 + |\z_*|^2 = |\z'|^2 + |\z_*'|^2$, we have
\begin{align*}
&\left| \Gamma_{\gain}(h_1, h_2)(x, \z) \right|\\
\leq& \int_{\mathbb{R}^{3}} \int_{0}^{2\pi} \int_{0}^{\frac{\pi}{2}} e^{-\frac{|\z_*|^2}{2}} |h_1(x, \z')| |h_2(x, \z'_{*})| B(| \z - \z_* |, \theta )\,d \theta d\phi d\z_*\\
\leq& \| h_1 \|_{L^\infty_\a(\O \times \R^3)} \| h_2 \|_{L^\infty_\a(\O \times \R^3)} \\
&\times \int_{\mathbb{R}^{3}}\int_{0}^{2\pi}\int_{0}^{\frac{\pi}{2}}e^{-\frac{|\z_*|^2}{2}}e^{-\alpha |\z'|^{2}} e^{-\alpha |\z_*'|^{2}}B(|\z-\z_*|, \theta )\,d \theta d\phi d\z_*\\
=& \| h_1 \|_{L^\infty_\a(\O \times \R^3)} \| h_2 \|_{L^\infty_\a(\O \times \R^3)} e^{-\alpha |\z|^{2}}\\ 
&\times \int_{\mathbb{R}^{3}}\int_{0}^{2\pi}\int_{0}^{\frac{\pi}{2}}e^{-\frac{|\z_*|^2}{2}}e^{-\alpha |\z_*|^2}B(|\z-\z_*|, \theta ) \,d\theta d\phi d\z_*\\
\leq& C \| h_1 \|_{L^\infty_\a(\O \times \R^3)} \| h_2 \|_{L^\infty_\a(\O \times \R^3)} e^{-\alpha |\z|^{2}}\\
&\times \int_{\mathbb{R}^{3}} e^{-\frac{|\z_*|^2}{2}}e^{-\alpha |\z_*|^2}|\z - \z_*|^{\gamma}\,d\z_* \int_{0}^{\frac{\pi}{2}} \sin{\theta}\cos{\theta}\,d\theta\\
\leq& C \| h_1 \|_{L^\infty_\a(\O \times \R^3)} \| h_2 \|_{L^\infty_\a(\O \times \R^3)} e^{-\alpha |\z|^{2}}(1+|\z|)^{\gamma}.
\end{align*}
Here, we used the estimate in Lemma \ref{lem:estimate_nonlin_1}.

In the same way, we obtain 
\[
\left|\Gamma_{\loss}(h_1, h_2)(x, \z) \right| \leq C \| h_1 \|_{L^\infty_\a(\O \times \R^3)} \| h_2 \|_{L^\infty_\a(\O \times \R^3)} e^{-\alpha |\z|^{2}} (1+|\z|)^{\gamma}.
\]

This completes the proof.
\end{proof}

Lemma \ref{lem:Gamma} implies that the function $\Gamma(h, h)$ belongs to $L^\infty_{\a,1}(\O \times \R^3)$ if $h$ belongs to $L^\infty_\a(\O \times \R^3)$.

We next give estimates for the $x$ derivatives of the nonlinear term.

\begin{lemma} \label{lem:Gamma_dx}
Let $\O$ be an open bounded convex domain with $C^2$ boundary of positive Gaussian curvature. Suppose \eqref{assumption_B1} holds. Then, for fixed $0 \leq \a < 1/2$, there exists some positive constant $C$ such that  for any $h_1, h_2 \in \cW_\a$, we have
\begin{align*}
|\nabla_x \Gamma(h_1, h_2)(x, \z)| \leq C \left(d_x^{-\frac{1}{2}} + w(x, \z)^{-1} \right) \| h_1 \|_{\cW_\a} \| h_2 \|_{\cW_\a} e^{-\alpha |\z|^2} (1 + |\z|)^\gamma
\end{align*}
for a.e. $(x, \z) \in \Omega \times \R^3$.
\end{lemma}

\begin{proof}
We first treat the gain term. We notice that
\[
\nabla_x\Gamma_{\gain}(h_1, h_2)(x, \z) = G_1(x, \z) + G_2(x, \z),
\]
where
\begin{align*}
G_1(x, \z) :=& \int_{\mathbb{R}^{3}} \int_{0}^{2\pi} \int_{0}^{\frac{\pi}{2}} e^{-\frac{|\z_*|^2}{2}} \left( \nabla_x h_1(x, \z') \right) h_2(x, \z'_{*}) B(|\z-\z_*|, \theta)\,d\theta d\phi d\z_*,\\
G_2(x, \z) :=& \int_{\mathbb{R}^{3}} \int_{0}^{2\pi} \int_{0}^{\frac{\pi}{2}}e^{-\frac{|\z_*|^2}{2}} h_1(x, \z') \left(\nabla_x h_2(x, \z'_{*}) \right)B(\vert v-\z_* \vert, \theta)\,d\theta d\phi d\z_*.
\end{align*}
Since $h_1, h_2 \in \cW_\a$, we have
\[
|\nabla_x h_j (x, \z)| \leq \| h_j \|_{\cW_\a} w(x, \z)^{-1} e^{-\alpha|\z|^{2}}, \quad j = 1, 2. 
\]
Thus, by Proposition \ref{prop:N1/2}, we have
\begin{align*}
|G_1 (x, \z)| \leq& \| h_1 \|_{\cW_\a} \| h_2 \|_{L^\infty_\a(\O \times \R^3)} e^{-\alpha|\z|^{2}}\\
& \times \int_{\mathbb{R}^{3}}\int_{0}^{2\pi}\int_{0}^{\frac{\pi}{2}}e^{-\frac{|\z_*|^2}{2}}e^{-\alpha|\z_*|^2} w(x, \z')^{-1} B(|\z-\z_*|, \theta) \,d\theta d\phi d\z_*\\
\leq& C d_x^{-\frac{1}{2}} \| h_1 \|_{\cW_\a} \| h_2 \|_{\cW_\a} e^{-\alpha|\z|^{2}}\\
& \times \int_{\mathbb{R}^{3}} \int_{0}^{2\pi} \int_{0}^{\frac{\pi}{2}}e^{-\frac{|\z_*|^2}{2}}e^{-\alpha|\z_*|^2} \left( 1 + \frac{1}{|\z'|} \right) B(|\z-\z_*|, \theta) \,d\theta d\phi d\z_*.
\end{align*}
In the same way, we have
\begin{align*}
|G_2 (x, \z)| \leq& C d_x^{-\frac{1}{2}} \| h_1 \|_{\cW_\a} \| h_2 \|_{\cW_\a} e^{-\alpha|\z|^{2}}\\
& \times \int_{\mathbb{R}^{3}} \int_{0}^{2\pi} \int_{0}^{\frac{\pi}{2}}e^{-\frac{|\z_*|^2}{2}}e^{-\alpha|\z_*|^2} \left( 1 + \frac{1}{|\z_*'|} \right) B(|\z-\z_*|, \theta) \,d\theta d\phi d\z_*.
\end{align*}
By Lemma \ref{lem:int_sigma} and Lemma \ref{lem:estimate_nonlin_1}, we obtain
\[
\int_{\mathbb{R}^{3}}\int_{0}^{2\pi}\int_{0}^{\frac{\pi}{2}}e^{-\frac{|\z_*|}{2}}e^{-\alpha|\z_*|^2} \left( 1+\frac{1}{|\z'|} \right) B(|\z - \z_*|, \theta )\,d \theta d\phi d\z_* \leq C (1 + |\z|)^\gamma
\]
and
\[
\int_{\mathbb{R}^{3}}\int_{0}^{2\pi}\int_{0}^{\frac{\pi}{2}}e^{-\frac{|\z_*|}{2}}e^{-\alpha|\z_*|^2} \left( 1+\frac{1}{|\z_*'|} \right) B(|\z - \z_*|, \theta )\,d \theta d\phi d\z_* \leq C (1 + |\z|)^\gamma.
\]
Hence, we have
\begin{align*}
\left| \nabla_x\Gamma_{\gain}(h_1, h_2)(x, \z) \right| \leq& |G_1(x, \z)| + |G_2(x, \z)|\\ 
\leq& C d_x^{-\frac{1}{2}} \| h_1 \|_{\cW_\a} \| h_2 \|_{\cW_\a} e^{-\alpha|\z|^{2}} (1 + |\z|)^\gamma
\end{align*}
for a.e. $(x, \z) \in \Omega \times \R^3$.

We next treat the loss term. We have
\begin{equation*}
\nabla_x\Gamma_{\loss}(h_1, h_2)(x, \z) = L_1(x, \z) + L_2(x, \z),
\end{equation*}
where
\begin{align*}
L_1(x, \z) :=& \left(\nabla_x h_1(x, \z) \right) \int_{\mathbb{R}^{3}} \int_{0}^{2\pi} \int_{0}^{\frac{\pi}{2}} e^{-\frac{|\z_*|^2}{2}} h_2(x, \z_*) B( |\z - \z_*|, \theta )\,d\theta d\phi d\z_*,\\
L_2(x, \z) :=& h_1(x, \z) \int_{\mathbb{R}^{3}} \int_{0}^{2\pi} \int_{0}^{\frac{\pi}{2}} e^{-\frac{|\z_*|^2}{2}} \left( \nabla_x h_2(x, \z_*) \right) B(|\z - \z_*|, \theta )\,d\theta d\phi d\z_*.
\end{align*}

For the $L_1$ term, by the assumption \eqref{assumption_B1}, Proposition \ref{prop:N1/2} and Lemma \ref{lem:estimate_nonlin_1}, we get
\begin{align*}
|L_1(x, \z)| \leq& C \| h_1 \|_{\cW_\a} \| h_2 \|_{L^\infty_\a(\O \times \R^3)} w(x, \z)^{-1} e^{-\alpha|\z|^{2}}\\
&\times \int_{\mathbb{R}^{3}} e^{-\frac{|\z_*|^2}{2}} e^{-\alpha |\z_*|^2} |\z - \z_*|^{\gamma}\,d\z_* \int_{0}^{2\pi}\int_{0}^{\frac{\pi}{2}}\sin{\theta}\cos{\theta}\,d\theta d\phi\\
\leq& C \| h_1 \|_{\cW_\a} \| h_2 \|_{\cW_\a} w(x, \z)^{-1} e^{-\alpha|\z|^{2}}(1 + |\z|)^\gamma
\end{align*}
for a.e. $(x, \z) \in \Omega \times \R^3$. For the $L_2$ term, applying Proposition \ref{prop:est_t+} and Lemma \ref{lem:estimate_nonlin_1}, we have
\begin{align*}
&|L_2(x, \z)|\\ 
\leq& C \| h_1 \|_{L^\infty_\a(\O \times \R^3)} \| h_2 \|_{\cW_\a} e^{-\alpha|\z|^{2}}\\
&\times \int_{\mathbb{R}^{3}}\int_{0}^{2\pi}\int_{0}^{\frac{\pi}{2}} e^{-\frac{|\z_*|^2}{2}} e^{-\alpha |\z_*|^2} w(x, \z_*)^{-1} B(|\z - \z_*|, \theta )\,d\theta d\phi d\z_*\\
\leq& C \| h_1 \|_{\cW_\a} \| h_2 \|_{\cW_\a} e^{-\alpha|\z|^{2}}  \int_{\mathbb{R}^{3}}e^{-\frac{|\z_*|^2}{2}}e^{-\alpha |\z_*|^2} w(x, \z_*)^{-1} |\z - \z_*|^{\gamma}\,d\z_*\\
\leq& C d_{x}^{-\frac{1}{2}} \| h_1 \|_{\cW_\a} \| h_2 \|_{\cW_\a} e^{-\alpha|\z|^{2}} \int_{\mathbb{R}^{3}} e^{-\frac{|\z_*|^2}{2}}e^{-\alpha |\z_*|^2} \ \left( 1 + \frac{1}{|\z_*|} \right) |\z - \z_*|^{\gamma} \,d\z_*\\
\leq& C d_{x}^{-\frac{1}{2}} \| h_1 \|_{\cW_\a} \| h_2 \|_{\cW_\a} e^{-\alpha|\z|^{2}} (1 + |\z|)^\gamma
\end{align*}
for a.e. $(x, \z) \in \Omega \times \R^3$. Thus, we obtain 
\begin{align*}
|\nabla_x\Gamma_{\loss}(h_1, h_2)(x, \z)| \leq C \left( d_{x}^{-\frac{1}{2}} + w(x, \z)^{-1} \right) \| h_1 \|_{\cW_\a} \| h_2 \|_{\cW_\a} e^{-\alpha|\z|^{2}}(1 + |\z|)^\gamma
\end{align*}
for a.e. $(x, \z) \in \Omega \times \R^3$.

The estimate for $|\nabla_x \Gamma(h_1, h_2)|$ is obtained by summing the estimates for $|\nabla_x \Gamma_{\gain}(h_1, h_2)|$ and $|\nabla_x \Gamma_{\loss}(h_1, h_2)|$.
\end{proof}

Finally we estimate the $\z$ derivative of the nonlinear term.

\begin{lemma} \label{lem:Gamma_dz}
Let $\O$ be an open bounded convex domain with $C^2$ boundary of positive Gaussian curvature. Suppose \eqref{assumption_B1}. Then, for $h_1, h_2 \in \tilde{\cW}_\a$, we have
\[
|\nabla_\z \Gamma(h_1, h_2)(x, \z)| \leq C \left( d_{x}^{-\frac{1}{2}} + w(x, \z)^{-1} \right) \| h_1 \|_{\tilde{\cW_\a}} \| h_2 \|_{\tilde{\cW}_\a} e^{-\alpha |\z|^2} (1 + |\z|)^\gamma
\]
for a.e. $(x, \z) \in \Omega \times \R^3$.
\end{lemma}

\begin{proof}
For the gain term, by the sigma formulation, we have
\begin{align*}
\int_{0}^{2\pi} \int_{0}^{\frac{\pi}{2}} h_1(\z') h_2(\z'_{*}) B(|\z - \z_*|, \theta)\,d \theta d\phi =& \int_{\mathbb{S}^2} h_1(\z') h_2(\z'_{*})\frac{B(|\z - \z_*|, \theta )}{2\sin{\theta}\cos{\theta}}\,d\Sigma(\sigma)\\
=& \frac{B_0}{2} \int_{\mathbb{S}^2} h_1(\z') h_2(\z'_{*}) |\z - \z_*|^\gamma\,d\Sigma(\sigma).
\end{align*}
Here, we used the assumption \eqref{assumption_B1}. Thus, we obtain
\begin{align*}
|\nabla_\z\Gamma_{\gain}(h_1, h_2)| \leq C \int_{\mathbb{R}^{3}}e^{-\frac{|\z_*|^2}{2}}\left|\nabla_\z\int_{\mathbb{S}^2} h_1(\z') h_2(\z'_{*})|\z - \z_*|^{\gamma}\,d\Sigma(\sigma) \right|\,d\z_*.
\end{align*}

Notice that, by the identity $|\z|^2 + |\z_*|^2 = |\z'|^2 + |\z_*'|^2$, we have
\begin{align*}
&\left| \nabla_\z \int_{\mathbb{S}^2} h_1(\z') h_2(\z'_{*}) |\z - \z_*|^{\gamma}\,d\Sigma(\sigma) \right|\\
\leq& \int_{\mathbb{S}^2} \left| \nabla_\z h_1(\z') \right| |h_2(\z'_{*})| |\z - \z_*|^{\gamma}\,d\Sigma(\sigma) + \int_{\mathbb{S}^2} |h_1(\z')| \left| \nabla_\z h_2(\z'_{*}) \right| |\z - \z_*|^{\gamma}\,d\Sigma(\sigma)\\
&+ \int_{\mathbb{S}^2} |h_1(\z')| |h_2(\z'_{*})| |\z - \z_*|^{\gamma-1}\,d\Sigma(\sigma)\\
\leq& C \| h_1 \|_{\tilde{\cW}_\a} \| h_2 \|_{L^\infty_\a(\O \times \R^3)} |\z - \z_*|^{\gamma} \int_{\mathbb{S}^2} w(x, \z')^{-1} e^{-\alpha(|\z'|^{2} + |\z_*'|^{2})} \,d\Sigma(\sigma) \\
&+ C \| h_1 \|_{L^\infty_\a(\O \times \R^3)} \| h_2 \|_{\tilde{\cW}_\a} |\z - \z_*|^{\gamma} \int_{\mathbb{S}^2} w(x, \z_*')^{-1} e^{-\alpha(|\z'|^{2} + |\z_*'|^{2})} \,d\Sigma(\sigma)\\
&+ C \| h_1 \|_{L^\infty_\a(\O \times \R^3)} \| h_2 \|_{L^\infty_\a(\O \times \R^3)} |\z - \z_*|^{\gamma-1} \int_{\mathbb{S}^2} e^{-\alpha(|\z'|^{2} + |\z_*'|^{2})} \,d\Sigma(\sigma)\\
\leq& C \| h_1 \|_{\tilde{\cW}_\a} \| h_2 \|_{\tilde{\cW}_\a} |\z - \z_*|^{\gamma} e^{-\alpha(|\z|^2 + |\z_*|^2)}\int_{\mathbb{S}^2} w(x, \z')^{-1} \,d\Sigma(\sigma) \\
&+ C \| h_1 \|_{\tilde{\cW}_\a} \| h_2 \|_{\tilde{\cW}_\a} |\z - \z_*|^{\gamma} e^{-\alpha(|\z|^2 + |\z_*|^2)} \int_{\mathbb{S}^2} w(x, \z_*')^{-1} \,d\Sigma(\sigma)\\
&+ C \| h_1 \|_{\tilde{\cW}_\a} \| h_2 \|_{\tilde{\cW}_\a} |\z - \z_*|^{\gamma-1} e^{-\alpha(|\z|^2 + |\z_*|^2)}.
\end{align*}
By Lemma \ref{lem:int_sigma}, we have
\[
\int_{\mathbb{S}^2} w(x, \z')^{-1} \,d\Sigma(\sigma) \leq C d_x^{-\frac{1}{2}} \left( 1+\frac{1}{|\z - \z_*|} \right)
\]
and
\[
\int_{\mathbb{S}^2} w(x, \z_*')^{-1} \,d\Sigma(\sigma)
\leq C d_x^{-\frac{1}{2}} \left( 1+\frac{1}{|\z - \z_*|} \right).
\]
Hence, we have
\begin{align*}
&\left|\nabla_\z \Gamma_{\gain}(h_1, h_2)(x, \z) \right|\\
\leq& C d_x^{-\frac{1}{2}} \| h_1 \|_{\tilde{\cW}_\a} \| h_2 \|_{\tilde{\cW}_\a} e^{-\alpha|\z|^{2}} \int_{\mathbb{R}^{3}}e^{-\frac{|\z_*|^2}{2}} |\z - \z_*|^{\gamma}e^{-\alpha|\z_*|^2} \left( 1+\frac{1}{|\z - \z_*|} \right) d\z_*\\
&+ \| h_1 \|_{\tilde{\cW}_\a} \| h_2 \|_{\tilde{\cW}_\a} e^{-\alpha|\z|^{2}}\int_{\mathbb{R}^{3}}e^{-\frac{|\z_*|^2}{2}}e^{-\alpha|\z_*|^2} |\z - \z_*|^{\gamma-1}d\z_*\\
\leq& C d_x^{-\frac{1}{2}} \| h_1 \|_{\tilde{\cW}_\a} \| h_2 \|_{\tilde{\cW}_\a} e^{-\alpha|\z|^{2}} (1 + |\z|)^\gamma
\end{align*}
for a.e. $(x, \z) \in \Omega \times \R^3$.

For the loss term, we notice that 
\begin{align*}
\Gamma_{\loss}(h_1, h_2) = h_1(\z) \int_{\mathbb{R}^{3}}e^{-\frac{|\z_*|^2}{2}} h_2(\z_*) \int_0^{2\pi} \int_0^{\frac{\pi}{2}} B(|\z - \z_*|, \theta )\,d\theta d\phi \,d\z_*,
\end{align*}
and, under the assumption \eqref{assumption_B1},
\[
\int_0^{2\pi} \int_0^{\frac{\pi}{2}} B(|\z - \z_*|, \theta )\,d\theta d\phi = B_0 \pi |\z - \z_*|^\gamma.
\]
Thus, the gradient of $\Gamma_{\loss}$ with respect to the $\z$ variable is described as below:
\begin{equation*}
\nabla_\z \Gamma_{\loss}(h_1, h_2)(x, \z) = L_3(x, \z) + L_4(x, \z),
\end{equation*}
where
\begin{align*}
L_3(x, \z) :=& B_0 \pi \left( \nabla_\z h_1(\z) \right) \int_{\mathbb{R}^{3}}e^{-\frac{|\z_*|^2}{2}} h_2(\z_*) |\z - \z_*|^\gamma \,d\z_*,\\
L_4(x, \z) :=& B_0 \gamma \pi h_1(\z) \int_{\mathbb{R}^{3}}e^{-\frac{|\z_*|^2}{2}} h_2(\z_*) |\z - \z_*|^{\gamma - 2}(\z - \z_*) d\z_*.
\end{align*}
For the $L_3$ term, we apply Lemma \ref{lem:estimate_nonlin_1} to obtain
\begin{align*}
|L_3(x, \z)|\leq& C \| h_1 \|_{\tilde{\cW}_\a} \| h \|_{L^\infty_\a(\O \times \R^3)} w(x, \z)^{-1} e^{-\alpha|\z|^{2}}\int_{\mathbb{R}^{3}} e^{-\frac{|\z_*|^2}{2}} e^{-\alpha|\z_*|^2}|\z - \z_*|^{\gamma}d\z_*\\
\leq& C \| h_1 \|_{\tilde{\cW}_\a} \| h_2 \|_{\tilde{\cW}_\a} w(x, \z)^{-1} e^{-\alpha|\z|^{2}}(1 + |\z|)^\gamma.
\end{align*}
For the $L_4$ term, we apply Lemma \ref{lem:estimate_nonlin_2} to get
\begin{align*}
|L_4(x, \z)| \leq& C \| h_1 \|_{L^\infty_\a(\O \times \R^3)} \| h_2 \|_{L^\infty_\a(\O \times \R^3)} e^{-\alpha|\z|^2} \int_{\mathbb{R}^{3}}e^{-\frac{|\z_*|^2}{2}}e^{-\alpha|\z_*|^2}| \z - \z_* |^{\gamma-1}\,d\z_*\\
\leq& C \| h_1 \|_{\tilde{\cW}_\a} \| h_2 \|_{\tilde{\cW}_\a} w(x, \z)^{-1} e^{-\alpha|\z|^{2}}(1 + |\z|)^\gamma.
\end{align*}
Hence we conclude that
\begin{align*}
|\nabla_\z \Gamma_{\loss}(h_1, h_2)(x, \z)| \leq C \| h_1 \|_{\tilde{\cW}_\a} \| h_2 \|_{\tilde{\cW}_\a} w(x, \z)^{-1} e^{-\alpha |\z|^{2}} (1 + |\z|)^\gamma
\end{align*}
for a.e. $(x, \z) \in \Omega \times \R^3$.

Combining the estimates for $|\nabla_\z \Gamma_{\gain}(h_1, h_2)|$ and $|\nabla_\z \Gamma_{\loss}(h_1, h_2)|$, we obtain the desired estimate.
\end{proof}

\section{$W^{1, p}$ estimates for the derivatives}

In Section 1, with the help of Theorem \ref{lem:sol_dz_lin} and Lemma \ref{lem88}, we  have showed the existence of a solution $f$ to the boundary value problem \eqref{BVP_non_red} in $\tilde{\cW}_\a$. In this section, we show that $\tilde{\cW}_\a \subset W^{1, p}(\O \times \R^3)$ for $0 < \a < 1/2$, which implies the second statement of Theorem \ref{main theorem 2}. 

Recall that, for $(x, \z) \in  \Omega \times (\R^3 \setminus \{0\})$, we have
\begin{align*}
|f(x, \z)| \leq& C e^{-\a |\z|^2},\\
|\nabla_x f(x, \z)| \leq& C e^{-\a |\z|^2} w(x, \z)^{-1},\\
|\nabla_\z f(x, \z)| \leq& C e^{-\a |\z|^2} w(x, \z)^{-1}
\end{align*}
for $f \in \tilde{\cW}_\a$ and some constant $C$. Thus, to prove the second assertion of Theorem \ref{main theorem 2}, it suffices to show the following lemma.

\begin{lemma} \label{LpN_ex}
Let $\O$ be an open bounded convex domain in $\R^3$ with $C^2$ boundary of positive Gaussian curvature. Also, let $w$ be defined by \eqref{def:w}. Then, $w^{-1} e^{-\a |\z|^2}$ belongs to $L^p(\Omega \times \mathbb{R}^3)$ for $1 \leq p < 3$ and for all $\a > 0$.
\end{lemma}

\begin{proof}
By Lemma \ref{changeintegration}, we have
\begin{equation*} \begin{split}
&\int_{\mathbb{R}^3} \int_\Omega \left( \frac{1}{w(x, \z)} e^{-\a |\z|^2} \right)^p\,dx d\z \\
=& \int_{\Gamma_-} \int_0^{\tau_+(y, \z)} \frac{1}{N(y + t \z, \z)^p} \left( \frac{1 + |\z|}{|\z|} \right)^p e^{-\a p |\z|^2} |n(y) \cdot \z| \,dt\,d\z\,d\sigma_y.
\end{split}\end{equation*}
Since
\[
N(y + t\z, \z) = N(y, \z) = \frac{|n(y) \cdot \z|}{|\z|},
\]
we have
\[
\int_0^{\tau_+(y, \z)} \frac{1}{N(y + t \z, \z)^p} \left( \frac{1 + |\z|}{|\z|} \right)^p |n(y) \cdot \z|\,dt = \frac{\tau_+(y, \z) (1 + |\z|)^p }{|\z|^{p - 1} N(y, \z)^{p - 1}}, 
\]
and hence we have
\begin{equation*} 
\begin{split}
&\int_{\Gamma_-} \int_0^{\tau_+(y, \z)} \frac{1}{N(y + t\z, \z)^p} \left( \frac{1 + |\z|}{|\z|} \right)^p e^{-\a p |\z|^2} |n(y) \cdot \z| \,dt\,d\z\,d\sigma_y \\
=& \int_{\partial \Omega} \int_{\Gamma_{-, y}} \frac{\tau_+(y, \z) (1 + |\z|)^p}{|\z|^{p - 1} N(y, \z)^{p - 1}} e^{-\a p |\z|^2}\,d\z\,d\sigma_y,
\end{split}
\end{equation*}
where
\[
\Gamma_{-, y} := \{ \z \in \mathbb{R}^3 \mid n(y) \cdot \z < 0 \}.
\]

Thanks to Proposition \ref{prop:est_t+}, we have
\[
\tau_+(y, \z) \leq C \frac{N(y, \z)}{|\z|}
\]
for all $(y, \z) \in \Gamma^-$. Thus, we have
\begin{align*}
&\int_{\partial \Omega} \int_{\Gamma_{-, y}} \frac{\tau_+(y, \z) (1 + |\z|)^p}{|\z|^{p - 1} N(y, \z)^{p - 1}} e^{-\a p |\z|^2}\,d\z\,d\sigma_y\\
\leq& C \int_{\partial \Omega} \int_{\Gamma_{-, y}} \frac{(1 + |\z|)^p}{|\z|^p N(y, \z)^{p - 2}} e^{-\a p |\z|^2}\,d\z\,d\sigma_y
\end{align*}

We are now in a position  to estimate the inner integral with respect to $\z$. Introducing the spherical coordinates, we have
\begin{align*}
&\int_{\Gamma_{-, y}} \frac{(1 + |\z|)^p}{|\z|^p N(y, \z)^{p - 2}} e^{-\a p |\z|^2}\,d\z\\ 
\leq& C \int_0^\infty \int_0^{2\pi} \int_0^{\frac{\pi}{2}} \frac{1 + \rho^p}{\rho^p \cos^{p-2} \theta } e^{-\a p \rho^2} \rho^2 \sin\theta \,d\theta \,d\varphi \,d\rho\\
=& C \left( \int_0^\infty \frac{1 + \rho^p}{\rho^{p - 2}} e^{-\a p \rho^2}\,d\rho \right) \left( \int_0^{\frac{\pi}{2}} \frac{\sin \theta}{\cos^{p-2} \theta}\,d\theta \right),
\end{align*}
whose right hand side is bounded for all $\a > 0$ if and only if $1 \leq p < 3$. This completes the proof.
\end{proof}

\section*{Data availability statement}
There is no associated scientific data used in this research.
\section*{Acknowledgement}

I-K. Chen was supported in part by NSTC with the grant number 108-2628-M-002-006-MY4 and 112-2115-M-002-009-MY3. C-H. Hsia was supported by National Science and Technology Council (NSTC) of Taiwan with grant number 112-2123-M-002-009 and National Center for Theoretical Sciences (NCTS) of Taiwan. D. Kawagoe was supported in part by JSPS KAKENHI grant number JP24K00539. The authors would like to express their gratitude to Mr.~Jin Zhi Phoong for his kind proofreading.

\end{document}